\documentclass{article}

 \usepackage[preprint]{neurips_2026}

\usepackage[utf8]{inputenc} 
\usepackage[T1]{fontenc}    
\usepackage{hyperref}       
\usepackage{url}            
\usepackage{booktabs}       
\usepackage{nicefrac}       
\usepackage{microtype}      
\usepackage{xcolor}         
\usepackage{bm}
\usepackage{amsmath,amsfonts,amsthm}
\usepackage{subcaption}
\usepackage{latexsym}
\usepackage{mathrsfs}
\usepackage{float}
\usepackage{graphicx}
\usepackage{cleveref}
\usepackage{enumitem}
\usepackage{tcolorbox}
\usepackage{array} 
\usepackage{mathtools, amssymb}

\newtheorem{lemma}{Lemma}
\newtheorem{theorem}{Theorem}
\newtheorem{corollary}{Corollary}
\newtheorem{definition}{Definition}
\newtheorem{proposition}{Proposition}

\newtheorem{assumption}{Assumption}
\newtheorem{fact}{Fact}
\newtheorem{remark}{Remark}
\newtheorem{exam}{Example}

\def\R{\mathbb{R}}
\def\P{\mathbb{P}}
\def\E{\mathbb{E}}
\def\mO{\mathcal{O}}
\newcommand{\st}{\mbox{s.t.}}
\newcommand{\bx}{\boldsymbol{x}}
\newcommand{\mN}{\mathcal{N}}
\newcommand{\bI}{\boldsymbol{I}}
\newcommand{\be}{\boldsymbol{e}}
\newcommand{\ba}{\boldsymbol{a}}
\newcommand{\bV}{\boldsymbol{V}}
\newcommand{\bU}{\boldsymbol{U}}
\newcommand{\bP}{\boldsymbol{P}}
\newcommand{\bQ}{\boldsymbol{Q}}

\newcommand{\revise}[1]{{\color{blue}#1}}
\title{A Concentration Inequality for the Covariance Matrix of an Arbitrary Subset of Random Vectors\thanks{Corresponding author: Peng Wang (Email: \texttt{pengw@um.edu.mo})}}

\author{%
  Huikang Liu\\
  Shanghai Jiao Tong University, Shanghai\\
  \texttt{ hkl1u@sjtu.edu.cn} \\
  \And
  Peng Wang \\
  University of Macau, Macao\\ 
  \texttt{pengw@um.edu.mo}
  \\
  \AND
  Laura Balzano \\
  University of Michigan, Ann Arbor\\
  \texttt{girasole@umich.edu}
}

\begin{document}

\maketitle

\vspace{-0.1in}
\begin{abstract}

Concentration inequalities for sample covariance matrices are fundamental tools in high-dimensional probability. Classical results typically assume that the selected random vectors are independent of the selection rule. In this paper, we study spectral concentration for sample covariance matrices formed from arbitrary, possibly data-dependent subsets of i.i.d. random vectors. Such data-dependent selection destroys the usual independence structure and makes standard covariance concentration bounds inapplicable. For i.i.d. Gaussian random vectors, we prove high-probability lower and upper bounds for the minimal and maximal eigenvalues of such selected covariance matrices. Compared with a direct union-bound argument, our results provide substantially sharper guarantees and allow much smaller subset proportions. We further discuss extensions from Gaussian to sub-Gaussian random vectors, and beyond independence to weakly dependent observations, with geometrically strong-mixing Gaussian sequences serving as a representative example of the latter. Finally, we apply the developed concentration inequalities to the K-subspace clustering problem under a low-rank Gaussian mixture model, where the optimal clusters are inherently data-dependent. Our results yield recovery guarantees showing that the clustering error of global minimizers decays polynomially with the signal-to-noise ratio.

\end{abstract}

\vspace{-0.1in}
\section{Introduction}\label{sec:intro}

Concentration inequalities for estimating the spectrum of sample covariance matrices of independent random vectors have been extensively studied in the literature and have found wide applications in statistics, signal processing, and machine learning \cite{ledoux2001concentration,tropp2015introduction,vershynin2018high}. Typically, these results assume that the random vectors in concentration inequalities are independent and identically distributed (i.i.d.).
A classic example of these concentration results is the spectral bound of the sample covariance matrix of i.i.d. Gaussian random vectors \citep[Theorem 4.7.1]{vershynin2018high},  a consequence of which is as follows\footnote{A proof of \Cref{lem:cov esti} can be found in \cite{so2022lecture}.}.  
\begin{fact}\label{lem:cov esti}
Suppose that $\{\bm a_i\}_{i=1}^M \subseteq \R^m$ are i.i.d. standard Gaussian random vectors. 
It holds for $\eta >0$ with probability at least $1-2\exp\left(-2\eta^2\right)$ that 
\begin{align}\label{eq:spec}
\left\| \frac{1}{M} \sum_{i=1}^M \bm a_i\bm a_i^T  -  \bm I_m \right\| \le \frac{9(\sqrt{m}+\eta)}{\sqrt{M}}\;.
\end{align} 
\end{fact}

However, the independence assumption may be violated in applications. A typical example is subspace clustering, a widely adopted method for identifying low-dimensional subspaces in high-dimensional datasets \cite{soltanolkotabi2014robust,vidal2011subspace}. Here, suppose that the data samples are generated by a low-rank Gaussian mixture model (see \Cref{def:MoG}), which is a commonly used model for theoretical analysis of the subspace clustering problem; see, e.g., \cite{soltanolkotabi2014robust,soltanolkotabi2012geometric,wang2022convergence}. 

When analyzing the optimal cluster assignments and the corresponding subspace bases of the objective function for this problem (see Problem \eqref{eq:loss}), the minimal eigenvalue of the sample covariance matrix in each optimal cluster assignment plays a core role in bounding the clustering error and subspace recovery error. However, one can observe that both the optimal cluster assignments and the subspace bases depend on the random vectors in the mixture model. Consequently, the selected index sets are statistically coupled with the data, so the selected samples cannot be treated as an independent sample drawn independently of the selection rule. Several related lines of work have extended concentration inequalities either to uniform control over adversarially selected large subsets \cite{diakonikolas2019robust,lecturenote}
or to certain dependent stochastic processes \cite{park2021estimating,paulin2015concentration,tropp2011freedman}. However, to the best of our knowledge, few results establish sharp spectral concentration inequalities for sample covariance matrices under the selection-induced dependence arising in the above setup. We defer a more detailed comparison to \Cref{sec:relatedwork}. 



\subsection{Our contributions.} To address this challenge, we study spectral concentration for sample covariance matrices formed from arbitrary, possibly data-dependent subsets of i.i.d. random vectors. Our contributions are threefold:

First, we establish concentration inequalities for the minimal and maximal eigenvalues of sample covariance matrices formed from arbitrary subsets of i.i.d. Gaussian random vectors; see \Cref{thm:lowerbound,thm:upperbound}. 
We also prove complementary bounds demonstrating their tightness; see \Cref{coro:lowerboundtightness}. 
Finally, we provide a sub-Gaussian extension under additional nondegeneracy and uniform tail assumptions; see \Cref{thm:lowerbound:general}.

Second, we extend our concentration framework beyond the independent setting. 
In particular, we show that the proof strategy remains valid for weakly dependent data as long as the key probabilistic ingredients are available, and we verify this principle for geometrically strong-mixing Gaussian sequences; see \Cref{thm:mixing-lower-bound}.  
This extension demonstrates that our approach can handle data-dependent covariance matrices arising from dependent observations.

Finally, we apply the established concentration inequalities to analyze the global optimality of the K-subspace clustering problem (see Problem \eqref{eq:loss}) under a low-rank Gaussian mixture model. 
By leveraging the concentration result in \Cref{thm:lowerbound}, we show that the clustering error of the global optimal solutions decreases polynomially as the signal-to-noise ratio (SNR) increases; see \Cref{thm:opti 1}.  



\subsection{Related Work}\label{sec:relatedwork}

\paragraph{Numerically erasure-robust frames (NERFs).}
A related line of work studies NERFs, i.e., frames for which every sufficiently large subframe obtained after arbitrary erasures remains uniformly well-conditioned. Fickus and Mixon \cite{fickus2012numerically} introduced this notion and showed that random Gaussian frames can tolerate roughly \(15\%\) arbitrary erasures. Diakonikolas et al. \cite{diakonikolas2019robust} and the related lecture note \cite{lecturenote} used concentration inequalities together with union bounds over subsets to obtain uniform covariance bounds; however, the resulting lower bounds are informative mainly when the retained fraction is close to one. Wang \cite{wang2018random} established covariance eigenvalue bounds that hold simultaneously for all admissible subsets, although the dependence of the constants on the retained proportion is relatively implicit. Han and Xu \cite{han2017robustness} derived correct-order estimates for Gaussian matrices under arbitrary row erasures. In comparison, we provide explicit finite-sample bounds expressed through Gaussian quantiles, together with complementary tightness results, and discuss extensions under additional assumptions to sub-Gaussian and strong-mixing observations, with an application to \(K\)-subspace clustering.

\vspace{-0.1in}
\paragraph{Concentration inequalities under dependence.} Classical concentration inequalities, such as those by \cite{tropp2015introduction,vershynin2018high}, assume independence and are widely applied in random matrix theory and high-dimensional statistics. Recently, several works have extended these inequalities to settings where random vectors have dependent structures, such as martingale differences \cite{tropp2011freedman}, Markov dependence \cite{paulin2015concentration}, and mixing-type dependence \cite{merlevede2009bernstein}. In particular, \cite{park2021estimating} analyzes an inverse probability weighted estimator for high-dimensional covariance and precision matrices in the presence of general dependent missingness. \cite{han2020moment} derives moment bounds for the deviation of large autocovariance matrices from their means under weak dependence. 

\vspace{-0.1in}
\paragraph{Subspace clustering.} Subspace clustering is a central problem in unsupervised learning, which aims to cluster data points that lie in a union of low-dimensional subspaces \cite{balzano2012k,soltanolkotabi2014robust,soltanolkotabi2012geometric,vidal2011subspace,wang2022convergence}. Notably, this problem can be solved by different approaches, including sparse subspace clustering \cite{Elhamifar2013Sparse,wang2013noisy}, spectral clustering \cite{pmlr-v139-li21f,vidal2005generalized}, and low-rank matrix factorization \cite{liu2013efficient,wang2022convergence}. To analyze these subspace clustering methods, particularly the non-convex ones, a union-of-subspaces model has become a standard data model for generating data samples \cite{soltanolkotabi2012geometric,soltanolkotabi2014robust,vidal2005generalized,wang2022convergence}. 

Among the approaches to spectral clustering, a popular one is the K-subspace method \cite{bradley2000k,lipor2021subspace}, which is a generalization of the K-means method to subspace clustering. Recently, its convergence behavior has been extensively studied under a mixture of low-rank Gaussians (see \Cref{def:MoG}). This model is a probabilistic instance of the union-of-subspaces model with Gaussian latent variables; see, e.g., \cite{lipor2021subspace,wang2022convergence}. Beyond subspace clustering, mixtures of low-rank Gaussians have also served as a tractable model for studying other unsupervised learning problems, including deep feature learning \cite{xu2026linearly} and diffusion models \cite{wang2025diffusion,pmlr-v235-zhang24cn}. Nevertheless, a complete characterization of the global optimality of the objective function (see Problem \eqref{eq:loss}) remains an open problem. The main challenge is that the optimal cluster assignments are data-dependent, so standard covariance concentration bounds for fixed subsets cannot be applied directly.

\vspace{-0.1in}
\paragraph{Notation.} 
For a vector $\bm a \in \R^d$, $\|\bm a\|$ denotes its $\ell_2$ norm, $a_i$ its $i$-th entry, and $\mathrm{diag}(\bm a)$ the diagonal matrix with $\bm a$ on the diagonal. 
For a matrix $\bm A \in \R^{m\times n}$, $\|\bm A\|$ or $\sigma_{\max}(\bm A)$ denotes its spectral norm (i.e., the largest singular value), $\sigma_{\rm min}(\bm A)$ the smallest singular value, $\|\bm A\|_F$ its Frobenius norm, and $a_{ij}$ its $(i,j)$-th entry. Given a symmetric matrix $\bm X \in \R^{n\times n}$, $\lambda_{\rm max}(\bm X)$ denotes its largest eigenvalue and $\lambda_{\rm min}(\bm X)$ denotes its smallest eigenvalue. Given a positive integer $n$, we denote by $[n]$ the set $\{1,\ldots,n\}$. Given a set of integers $\{n_k\}_{k=1}^K$, let $n_{\max} = \max\{n_k: k \in [K]\}$ and $n_{\min} = \min\{n_k: k \in [K]\}$.  We use \(\mathcal O^{n\times d}:=\{\bm U\in\mathbb R^{n\times d}:\bm U^T\bm U=\bm I_d\}\) for the Stiefel manifold and \(\mathcal O(d):=\mathcal O^{d\times d}\) for the orthogonal group.

Let $\{a_n\}$ and $\{b_n\}$ be real sequences with $b_n \neq 0$ eventually.  
We write $a_n = o(b_n)$ as $n \to \infty$ if $\lim_{n \to \infty}  {a_n}/{b_n} = 0$. Given a standard Gaussian random variable $X \sim \mathcal{N}(0,1)$, we denote its cumulative distribution function by  $\Phi(x) := 1/{\sqrt{2\pi}} \int_{-\infty}^x \exp\left(- {t^2}/{2}\right) dt.$ For $u \in (0,1)$, we define $\Phi^{-1}(u) := \inf\{ x\in\mathbb{R} : \Phi(x) \ge u \}.$ We use $X_n \overset{d}{\to} X$ to denote convergence in distribution.

\section{Main Results}\label{sec:main}
In this section, we establish concentration inequalities for data-dependent subsets of Gaussian random vectors in \Cref{subsec:main}, and extend them to sub-Gaussian random vectors in \Cref{subsec:subgauss}.


\subsection{Concentration Inequalities}\label{subsec:main}

We first present a concentration inequality that bounds the minimal eigenvalue of the sample covariance matrix formed from an arbitrary subset of Gaussian random vectors, even when the subset is chosen in a data-dependent manner.
\begin{theorem}\label{thm:lowerbound}
Let $\{\bm a_i\}_{i=1}^M \subseteq \R^m$ be a set of i.i.d. standard Gaussian random vectors, i.e., $\bm a_i \overset{i.i.d.}{\sim} \mathcal{N}(\bm 0, \bm I_m)$. 
Fix  $\kappa_1\in(0,1]$ and $\eta \in (0,1/2)$, and define
\begin{align} \label{eq:gamma mu}
\gamma := \Phi^{-1}\left( \frac{1 + (1 - \eta)\kappa_1}{2} \right),
\quad
\mu := (1-\eta)\kappa_1
- \sqrt{\frac{2\gamma^2}{\pi}}\exp\left(-\frac{\gamma^2}{2}\right).
\end{align}
Then there exists an absolute constant $c_1>0$ such that with probability at least
\begin{align}\label{eq:prob}
     1 - 2 \left( \frac{22}{\eta\sqrt{\mu }} + 1 \right)^m \left( \exp\left( -\frac{\eta^2 \kappa_1 M}{2} \right)  +  \exp\left( - c_1\eta^2\mu^2 M \right) \right) - 2\exp(-2M),
\end{align}
it holds for any subset $\Omega\subseteq[M]$ satisfying $|\Omega| \ge \kappa_1 M$ that 
\begin{align}\label{eq:spec min}
\lambda_{\min}\left( \sum_{i \in \Omega} \bm a_i\bm a_i^T\right) \ge 
(1 - \eta)^2 \mu M.
\end{align}
\end{theorem}
We defer the proof to \Cref{sec:proof-lowerbound}. Now, let us make some remarks on this concentration inequality.  
\begin{itemize}[leftmargin=*]
\item[$\bullet$] Even though $\{\bm a_i\}_{i=1}^M \subseteq \R^m$ are i.i.d. standard Gaussian random vectors, we should point out that $\Omega \subseteq [M]$ may depend on $\{\bm a_i\}_{i=1}^M$. For example, we define 
\begin{align*}
    \Omega = \left\{ i \in [M]:\|\bm a_i\| \ge \|\bm a\|_{(\kappa_1M)} \right\},
\end{align*}
where $\|\bm a\|_{(k)}$ denotes the $k$-th largest value among $\{\|\bm a_j\|\}_{j=1}^M$. Consequently, $\bm a_i$ for all $i\in \Omega$ depend on both itself and all the other vectors.   



\item[$\bullet$] It follows from \eqref{eq:gamma mu} and \eqref{eq6:lem cov A1 1} in \Cref{sec:proof-lowerbound} that 
\begin{align*}
    \mu = \sqrt{\frac{2}{\pi}}\int_0^{\gamma} x^2 \exp\left( -\frac{x^2}{2} \right) dx.  
\end{align*}
This, together with the fact that $\gamma$ defined in \eqref{eq:gamma mu} is monotone increasing w.r.t. $\kappa_1$, yields that the lower bound in \eqref{eq:spec min} increases as $\kappa_1$ increases.

\item[$\bullet$] We do not state \(M\ge m\) or \(\kappa_1M\ge m\) as separate assumptions, since these rank-feasibility requirements are already reflected by the probability lower bound: whenever \eqref{eq:prob} is non-vacuous, the first exponential term forces \(\kappa_1M\) to be at least of order \(m\). Beyond this necessary regime, the probability bound becomes numerically meaningful only when \(M/m\) is sufficiently large relative to \(\kappa_1\). As shown in \Cref{app:kappa}, for fixed \(M/m\), the smallest practically admissible \(\kappa_1\) scales roughly as
\[
    \kappa_1 \gtrsim
    \left(\log(M/m) /(M/m)\right)^{1/6}.
\]
The following example illustrates this dependence numerically and shows how larger sampling ratios \(M/m\) allow smaller admissible choices of \(\kappa_1\) while keeping the probability bound strong.

\end{itemize}

\begin{exam}\label{exam:parameter}
According to \Cref{thm:lowerbound}, the ratio $M / m$ is the key point to guarantee the result holds with high probability. 
\begin{itemize}[leftmargin=*]
\item If \(M=10^3m\), choosing \(\kappa_1\geq 0.7\) and \(\eta\geq 0.4\) gives \(\sqrt{\mu}\approx 0.20\), and the event in \eqref{eq:spec min} holds with probability at least \(0.999\).
\item If \(M=10^4m\), choosing \(\kappa_1\geq 0.3\) and \(\eta\geq 0.26\) gives \(\sqrt{\mu}\approx 0.08\), with probability at least \(0.999\).
\item If \(M=10^6m\), choosing \(\kappa_1\geq 0.05\), \(\eta\geq 0.2\) gives \(\sqrt{\mu}\approx 0.006\), with probability at least \(0.999\).
\end{itemize}
\end{exam}

\begin{figure*}[t]
    \centering
    \begin{subfigure}[b]{0.48\textwidth}
    \IfFileExists{fig/Mm1e4.pdf}{\includegraphics[width=\textwidth]{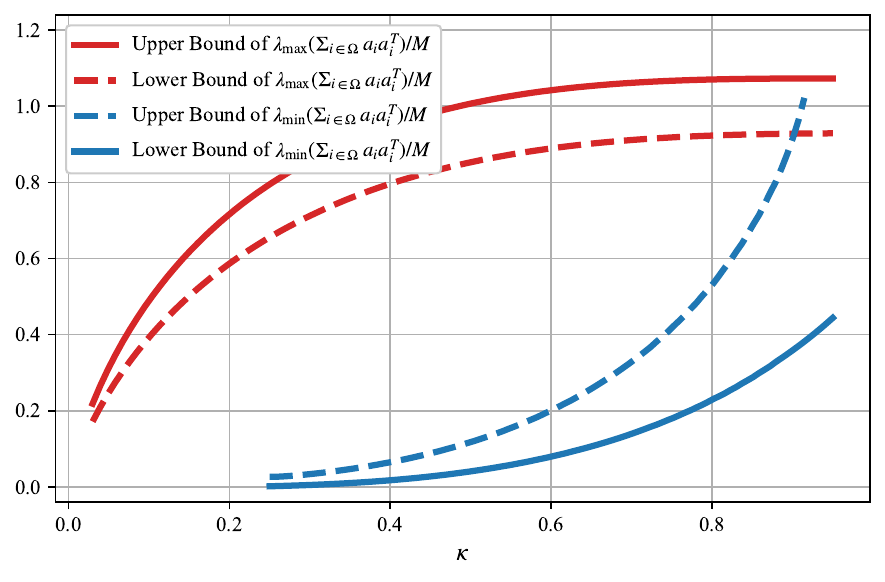}}{\fbox{\parbox[c][1.5in][c]{0.95\textwidth}{\centering\revise{Missing figure file: fig/Mm1e4.pdf}}}}\vspace{-0.1in} 
        \label{fig:image1}
    \end{subfigure}
    \hfill
    \begin{subfigure}[b]{0.48\textwidth}
        \IfFileExists{fig/Mm1e5.pdf}{\includegraphics[width=\textwidth]{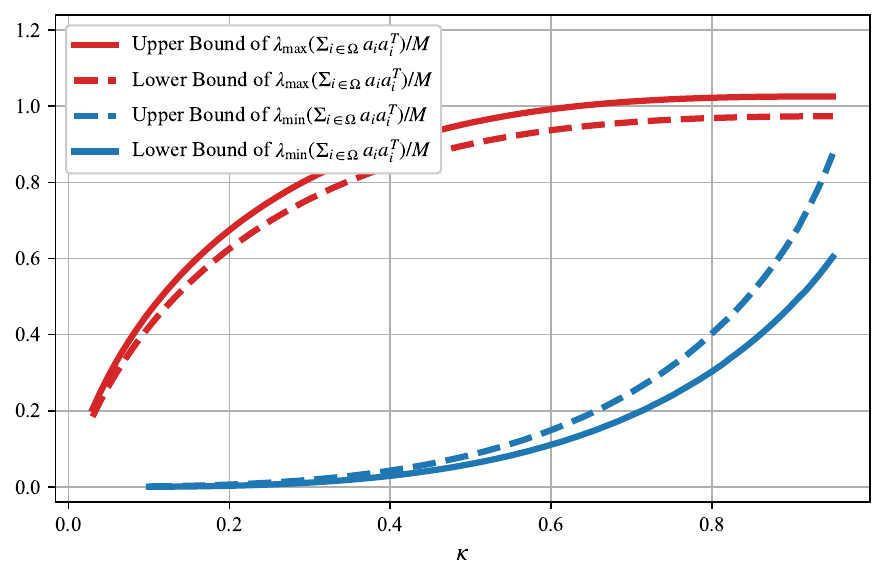}}{\fbox{\parbox[c][1.5in][c]{0.95\textwidth}{\centering\revise{Missing figure file: fig/Mm1e5.pdf}}}}\vspace{-0.1in}
        \label{fig:image2}
    \end{subfigure}\vspace{-0.15in}
\caption{Lower and upper bounds on $\lambda_{\min}\left( \sum_{i \in \Omega} \bm a_i\bm a_i^T\right)$ (see \eqref{eq:spec min} and \eqref{eq:spec min upper}) and $\lambda_{\max}\left( \sum_{i \in \Omega} \bm a_i\bm a_i^T\right)$ (see \eqref{eq:spec max} and \eqref{eq:spec max lower}) as functions of $\kappa$ (corresponding to $\kappa_1$ and $\kappa_2$, respectively).
The parameter $\eta$ is chosen such that all probabilities in \eqref{eq:prob}, \eqref{eq:prob 4}, \eqref{eq:prob 3}, and \eqref{eq:prop 2} are no less than $0.999$.
Results are shown for different ratios $M/m$ (\textbf{Left}: $M/m = 10^4$; \textbf{Right}: $M/m = 10^5$).}\label{fig:1}\vspace{-0.1in}
\end{figure*}

\vspace{-0.1cm}
\paragraph{Weakness of Simple Union Bound.} One may wonder whether standard concentration inequalities such as \Cref{lem:cov esti}, combined with a union bound, are sufficient to establish a bound of the form \eqref{eq:prob}. While this approach is indeed possible, as shown in \cite{lecturenote}, the resulting lower bound is meaningful only when the retained fraction is close to one. In particular, the union-bound argument imposes a stringent requirement on the parameter $\kappa_1$, namely $\kappa_1 > 0.78$, below which the resulting bound becomes vacuous in the asymptotic regime.

Consider the case \(m=1\) and a fixed subset 
\(\Omega\) with \(|\Omega|=\kappa_1M\). Define $Z_\Omega
:=
\frac{1}{\sqrt{2\kappa_1M}}
\sum_{i\in\Omega}(a_i^2-1).$
By the central limit theorem and Mills' ratio, for any fixed 
\(\eta\in(0,1/2]\) and all sufficiently large \(M\),
\[
\mathbb P\big(|Z_\Omega|\le \eta\big)
\le
1-c\,\frac{\eta}{1+\eta^2}\exp\left(-\frac{\eta^2}{2}\right),
\]
where \(c>0\) is an absolute constant. To make the fixed-subset covariance 
lower bound non-vacuous, one must take 
\(\eta\le \sqrt{\kappa_1M/2}\).
Taking a union bound over all \(C_{M}^{\kappa_1 M} \) subsets and using 
Stirling's approximation gives a resulting probability bound of at most
\[
1-\frac{c}{M}
\exp\left\{
M\left(
\kappa_1\log\frac{1}{\kappa_1}
+(1-\kappa_1)\log\frac{1}{1-\kappa_1}
-\frac{\kappa_1}{4}
\right)\right\}.
\]
Hence, for this bound to be non-vacuous as \(M\to\infty\), it is necessary $\kappa_1>e^{-1/4}\approx 0.78 .$ See \Cref{app:union} for more details. In contrast, our bound allows much smaller \(\kappa_1\), subject to the non-vacuous-probability regime described in \Cref{app:kappa}, making it much more flexible.

We next derive an upper bound on the maximum eigenvalue of $\sum_{i \in \Omega} \bm a_i\bm a_i^T$, as stated in the following theorem. The proof is deferred to \Cref{sec:proof-upperbound}.
\begin{theorem}\label{thm:upperbound}
    Let $\{\bm a_i\}_{i=1}^M \subseteq \R^m$ be a set of i.i.d. standard Gaussian random vectors, i.e., $\bm a_i \overset{i.i.d.}{\sim} \mathcal{N}(\bm 0, \bm I_m)$. Fix $\kappa_2\in(0,1]$ and $\eta\in(0,1/2)$ such that
$(1+\eta)\kappa_2\le 1$, and define
\begin{align*}
\gamma := \Phi^{-1}\left( 1 - \frac{(1+\eta)\kappa_2}{2} \right),
\quad
\mu := (1+\eta)\kappa_2
+ \sqrt{\frac{2}{\pi}}\,\gamma\exp\left(-\frac{\gamma^2}{2}\right).
\end{align*} 
Then there exists an absolute constant $c_2>0$ such that with probability at least 
\begin{align}\label{eq:prob 4}
    1 - 2\left( \frac{6}{\eta} + 1 \right)^m\left( \exp \left(-\frac{\eta^2\kappa_2 M}{6}\right)  + \exp\left(- c_2 \eta^2\mu^2 M \right) \right), 
\end{align}
it holds for every subset $\Omega \subseteq [M]$ satisfying $|\Omega| \leq \kappa_2 M$ with \(\kappa_2\in(0,1]\) that  
\begin{align}\label{eq:spec max}  \lambda_{\max}\left(\sum_{i \in \Omega} \bm a_i\bm a_i^T\right) \le (1 + \eta)^2\mu M. 
\end{align} 
\end{theorem}

\paragraph{Tightness of lower and upper bounds.} A natural question is the tightness of our bounds in \Cref{thm:lowerbound} and \Cref{thm:upperbound}.  To address this concern, we present an upper bound for $\lambda_{\min}\left( \sum_{i \in \Omega} \bm a_i\bm a_i^T\right)$ and a lower bound for $\lambda_{\max}\left( \sum_{i \in \Omega} \bm a_i\bm a_i^T\right)$ as follows:
\begin{corollary}\label{coro:lowerboundtightness}
 Let $\{\bm a_i\}_{i=1}^M \subseteq \R^m$ be a set of i.i.d. standard Gaussian random vectors.\\
 (i) With probability at least 
\begin{align}\label{eq:prob 3}
     1- 2\exp\left( - \eta^2\kappa_1 M / 4 \right) - 2\exp\left( - { c_1\eta^2 \widetilde{\mu}^2 M} \right),
\end{align}
there exists some $\Omega \subseteq [M]$ satisfying $|\Omega| \ge \kappa_1 M$ such that 
\begin{align}\label{eq:spec min upper}
 \lambda_{\min}\left( \sum_{i \in \Omega} \bm a_i\bm a_i^T\right) \le 
(1 + \eta) \widetilde{\mu} M,
\end{align}
where $\kappa_1 \in (0,1]$, $ c_1 > 0$, and $\eta \in (0, 1/2)$ are defined in \Cref{thm:lowerbound}, and 
\begin{align*} 
 \widetilde{\gamma} := \Phi^{-1}\left( \frac{1 + (1 + \eta)\kappa_1}{2} \right), \quad
 \widetilde{\mu} \coloneqq (1 +\eta)\kappa_1 - \sqrt{\frac{2\widetilde{\gamma}^2}{\pi}}\exp\left( -\frac{\widetilde{\gamma}^2}{2} \right). 
\end{align*}
(ii) With probability at least  
\begin{align}\label{eq:prop 2}
    1 - 2\exp \left(-\frac{\eta^2\kappa_2 M}{2}\right)  - 2\exp\left(- c_2 \cdot\eta^2\widetilde{\mu}^2 M \right), 
\end{align}
there exists some $\Omega \subseteq [M]$ satisfying $|\Omega| \leq \kappa_2 M$ such that  
\begin{align}\label{eq:spec max lower}
\lambda_{\max}\left( \sum_{i \in \Omega} \bm a_i\bm a_i^T\right) \ge (1 - \eta) \widetilde{\mu} M, 
\end{align} 
where \(\kappa_2\in(0,1]\), $c_2 > 0$, and $\eta \in (0, 1/2)$ are defined in \Cref{thm:upperbound}, and  
\begin{align*}
\widetilde{\gamma}  := \Phi^{-1}\left( 1 - \frac{(1-\eta)\kappa_2}{2} \right),\quad 
      \widetilde{\mu}  := (1-\eta)\kappa_2 + \sqrt{\frac{2\widetilde{\gamma}^2}{\pi}}\exp\left( -\frac{\widetilde{\gamma}^2}{2} \right). 
\end{align*}
\end{corollary}

\begin{figure*}[t]
    \centering
    \begin{subfigure}[b]{0.48\textwidth}
    \IfFileExists{fig/kappa04.pdf}{\includegraphics[width=\textwidth]{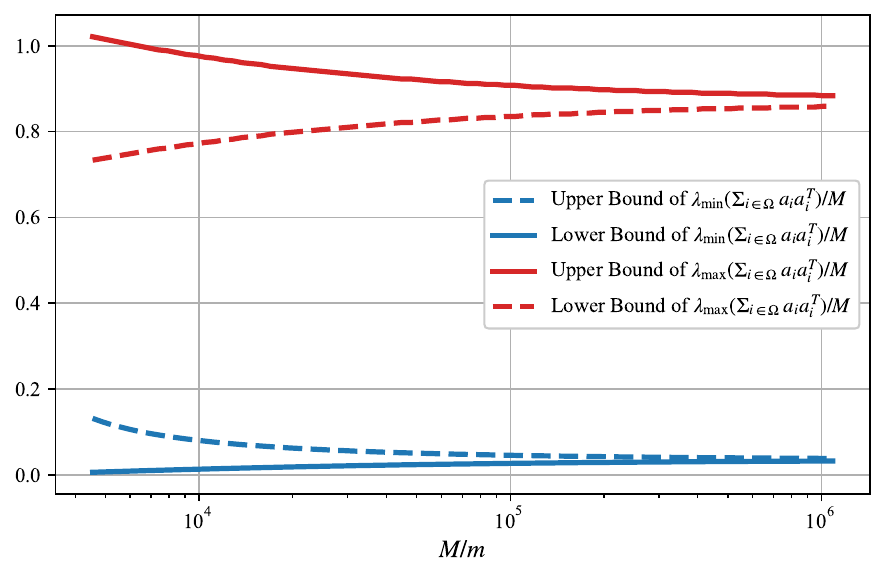}}{\fbox{\parbox[c][1.5in][c]{0.95\textwidth}{\centering\revise{Missing figure file: fig/kappa04.pdf}}}}\vspace{-0.1in} 
    \end{subfigure}
    \hfill
    \begin{subfigure}[b]{0.48\textwidth}
        \IfFileExists{fig/kappa06.pdf}{\includegraphics[width=\textwidth]{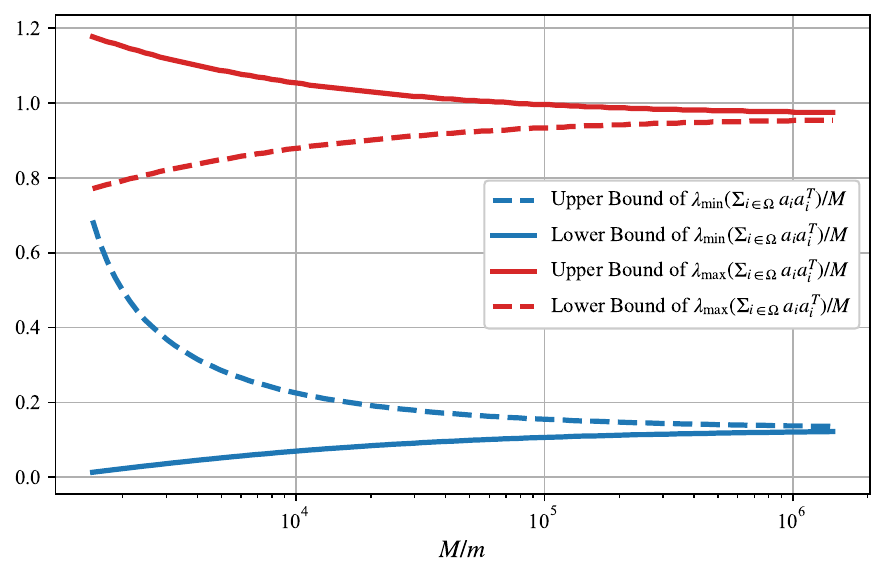}}{\fbox{\parbox[c][1.5in][c]{0.95\textwidth}{\centering\revise{Missing figure file: fig/kappa06.pdf}}}}\vspace{-0.1in}
    \end{subfigure}
\caption{Lower and upper bounds on $\lambda_{\min}\left( \sum_{i \in \Omega} \bm a_i\bm a_i^T\right)$ (see \eqref{eq:spec min} and \eqref{eq:spec min upper}) and $\lambda_{\max}\left( \sum_{i \in \Omega} \bm a_i\bm a_i^T\right)$ (see \eqref{eq:spec max} and \eqref{eq:spec max lower}) as functions of $M/m$.
The parameter $\eta$ is chosen such that all probabilities in \eqref{eq:prob}, \eqref{eq:prob 4}, \eqref{eq:prob 3}, and \eqref{eq:prop 2} are no less than $0.999$.
Results are shown for different $\kappa$ (Left: $\kappa = 0.4$; Right: $\kappa = 0.6$).}\vspace{-0.1in}
    \label{fig:2}
\end{figure*}

The proof is deferred to \Cref{app subsec:coro}. To further illustrate the tightness of the bounds established in \Cref{thm:lowerbound} and \Cref{thm:upperbound}, we plot the corresponding lower and upper bounds of $\lambda_{\min}\left( \sum_{i \in \Omega} \bm a_i \bm a_i^T \right)$ given in \eqref{eq:spec min} and \eqref{eq:spec min upper}, as well as those of
$\lambda_{\max}\left( \sum_{i \in \Omega} \bm a_i \bm a_i^T \right)$ given in \eqref{eq:spec max} and \eqref{eq:spec max lower}, in \Cref{fig:1} and \ref{fig:2}. Specifically, we first fix the ratio $M/m = 10^4, 10^5$ and examine how the bounds vary with respect to $\kappa$ in \Cref{fig:1}, and then fix $\kappa = 0.4, 0.6$ to investigate their dependence on $M/m$ in \Cref{fig:2}.
In all cases, the parameter $\eta$ is chosen to ensure that the probabilities in \eqref{eq:prob}, \eqref{eq:prob 4}, \eqref{eq:prob 3}, and \eqref{eq:prop 2} are no less than $0.999$. It is observed from these figures that the theoretical upper and lower bounds closely track the corresponding spectral quantities across a wide range of $\kappa$ and $M/m$. This indicates that the bounds derived in \Cref{thm:lowerbound} and \Cref{thm:upperbound} are not only valid but also quantitatively tight.


\subsection{Extension to Sub-Gaussian Random Vectors}\label{subsec:subgauss}

In this subsection, we extend the concentration inequalities in \Cref{subsec:main} from Gaussian random vectors to sub-Gaussian random vectors \citep[Definition~2.5.6]{vershynin2018high}. The proof is deferred to \Cref{sec:proof-lowerbound-general}.

\begin{theorem}\label{thm:lowerbound:general}
Let $\{\bm a_i\}_{i=1}^M \subseteq \R^m$ be a set of i.i.d. sub-Gaussian random vectors with $M \geq m$, and $P_{x}(\cdot)$ denotes the cumulative distribution function (CDF) of $\bm{a}_i^T \bm x$ for any $\bm x \in \mathbb{S}^{m-1}$. 
Define 
\begin{align} 
\Psi_1(z) := & \max_{x \in \mathbb{S}^{m-1}} \left\{ P_x(z) - P_x(-z) \right\}, \quad \forall z \geq 0, \label{eq:Psi1}\\
\Psi_2(z) := & \min_{x \in \mathbb{S}^{m-1}} \left\{ P_x(z) - P_x(-z) \right\}, \quad \forall z \geq 0. \label{eq:Psi2}
\end{align}
(i) Fix $\kappa_3\in(0,1]$ and $\eta\in(0,1/2)$. Let $\gamma>0$ satisfy $\Psi_1(\gamma)\le (1-\eta)\kappa_3,$ and define
\begin{align}\label{eq:gamma sub 1}
\mu
:= \min_{\bm x \in \mathbb{S}^{m-1}}
\mathbb{E}_{z \sim P_x}
\left[z^2 \bm{1}\{z^2 \leq \gamma^2\}\right].
\end{align}
There exists a constant $c_3 > 0$ such that with probability at least 
\begin{align*}
    & 1 -  2 \left( \frac{16c_3}{\eta\sqrt{\mu } } + 1 \right)^m \left( \exp\left( -\frac{\eta^2\kappa_3 M}{2} \right)  + \exp\left( - \frac{\eta^2\mu^2 M}{2\gamma^4} \right) \right) - 2\exp(-M),
\end{align*}
it holds for every subset $\Omega \subseteq [M]$ satisfying $|\Omega| \ge \kappa_3 M$ that 
\begin{align}\label{eq:spec min sub}
\lambda_{\min}\left( \sum_{i \in \Omega} \bm a_i\bm a_i^T\right) \ge 
(1 - \eta)^2 \mu M.
\end{align}
(ii) Fix $\kappa_4\in(0,1]$ and $\eta\in(0,1/2)$. Let $\gamma>0$ satisfy $\Psi_2(\gamma)\ge (1+\eta)\kappa_4,$ and define
\begin{align}\label{eq:gamma sub 2}
\mu
:= \max_{\bm x \in \mathbb{S}^{m-1}}
\mathbb{E}_{z \sim P_x}
\left[z^2 \bm{1}\{z^2 \geq \gamma^2\}\right].
\end{align} 
Then there exists a constant $c_4>0$ such that with probability at least 
\begin{align*}
    1 - 2\left( \frac{6}{\eta} + 1 \right)^m\left( \exp \left(-\frac{\eta^2\kappa_4 M}{6}\right)  + \exp\left(- c_4 \eta^2\mu^2 M \right) \right),
\end{align*}
it holds for every subset $\Omega \subseteq [M]$ satisfying $|\Omega| \le \kappa_4 M$ with $\kappa_4 \in [0,1]$ that 
\begin{align}\label{eq:spec max sub}
\left\| \sum_{i \in \Omega} \bm a_i\bm a_i^T\right\| \le 
(1 + \eta)^2 \mu M.
\end{align}
\end{theorem}

\section{Extension to Weakly Dependent Random Vectors}\label{sec:ext}

In this section, we discuss extensions beyond the independent setting. The main observation is that independence is not essential to the geometric component of the proof. Instead, the proof of \Cref{thm:lowerbound} only requires three probabilistic ingredients: a global operator-norm bound for the full data matrix and two fixed-direction concentration bounds\footnote{In the proof of \Cref{thm:lowerbound}, we provide the global operator-norm bound in \eqref{eq:aainormbound} and the fixed-direction concentration bounds in \eqref{ineq:bernstein bernoulli} and \eqref{ineq:subexponential-lowerbound}.}. Consequently, the same proof strategy extends to weakly dependent sequences \(\{\bm a_i\}_{i=1}^M\), as long as these three ingredients can be established.

To make this extension concrete, we consider geometrically strong-mixing sequences \cite{rosenblatt1956central} as a representative class of weakly dependent data. This class is broad enough to capture many temporally dependent processes while still admitting Bernstein-type concentration inequalities needed in our proof. Let $\{X_t\}_{t\in\mathbb Z_+}$ be a stochastic process and define
\[
\mathcal F_{0}^{\tau}:=\sigma(X_t: 0 \leq t \leq \tau),
\quad
\mathcal F_{\tau}^{\infty}:=\sigma(X_t:t\ge \tau).
\]
The mixing coefficients of $\{X_t\}_{t\in\mathbb Z_+}$ are defined by
\[
\alpha(k) := \sup_{t\in\mathbb Z_+}\sup_{A\in\mathcal F_{0}^{t},\,B\in\mathcal F_{t+k}^{\infty}}
\left|\mathbb P(A\cap B)-\mathbb P(A)\mathbb P(B)\right|, \quad k=1,2,\dots.
\]
\begin{definition}
    A sequence \(\{X_t\}_{t\in\mathbb Z}\) is called {\em geometrically strong mixing} if its mixing coefficients decay geometrically, i.e., there exist some constants \(C>0\) and \(\rho\in(0,1)\) such that
\[
\alpha(k) \le C\rho^k,\quad \forall k\ge 1.
\] 
\end{definition}

The following example shows that geometrically strong mixing holds for stationary Gaussian processes with sufficiently fast covariance decay; see \Cref{prop:gaussian-cov-decay-mixing} in \Cref{app:ext 1} for a rigorous statement and proof.
\begin{exam}\label{exam:station}
A standard stationary Gaussian sequence with geometrically decaying covariance is geometrically strong mixing. Specifically, suppose there exist constants \(C>0\) and \(\rho\in(0,1)\) such that, for all \(k,t\in\mathbb Z_+\),
\[
\bm a_t \sim \mathcal N(\bm 0,\bm I_m),\quad
\mathrm{Cov}(\bm a_0,\bm a_k)
=
\mathrm{Cov}(\bm a_t,\bm a_{t+k}),\quad
\|\mathrm{Cov}(\bm a_0,\bm a_k)\|\le C\rho^k.
\]
Then \(\{\bm a_t\}_{t\in\mathbb Z_+}\) is geometrically strong mixing.
\end{exam}

A standard consequence of geometric strong mixing is that bounded functions of
the process satisfy Bernstein-type inequalities with an effective sample size
loss of logarithmic order. More precisely, throughout this section, we use the
notation
\[
L_M:=\log(eM)\log\log(e^eM),
\quad
M_{\rm eff}:=M / L_M.
\]
The following theorem gives a dependent version of the lower spectral bound, whose proof is deferred to \Cref{app:ext 2}.

\begin{theorem}
\label{thm:mixing-lower-bound}
Let $\{\bm a_i\}_{i=1}^{M}\subseteq \mathbb R^m$ be a standard Gaussian sequence satisfying geometrically strong-mixing property with coefficient $C >0$ and $\rho \in (0,1)$. Fix $\eta\in(0,1/2)$ and $\kappa\in(0,1]$, and define
\[
\gamma := \Phi^{-1}\left( \frac{1 + (1-\eta)\kappa}{2} \right),\quad
\mu := (1-\eta)\kappa-\sqrt{\frac{2\gamma^2}{\pi}}\exp\left(-\frac{\gamma^2}{2}\right).
\]
With probability at least
\[
1-\exp(-c_1M)
-
2\left(
\frac{6\sqrt B}{\eta\sqrt\mu}+1
\right)^m
\left(
\exp\left(-c\eta^2\kappa^2 M_{\rm eff}\right)
+
\exp\left(
-\frac{c\eta^2\mu^2}{\gamma^4}M_{\rm eff}
\right)
\right),
\]
for some constants $c_1, B$ depending on $C$ and $\rho$, and an absolute constant \(c>0\), it holds simultaneously for every subset $\Omega\subseteq[M]$ satisfying
$|\Omega|\ge \kappa M$ that
\[
\textstyle  \lambda_{\min}
\left(
\sum_{i\in\Omega}\bm a_i \bm a_i^T
\right)
\ge
(1-\eta)^2 \mu M.
\]
\end{theorem}

\section{Application to K-Subspace Clustering}\label{sec:KSS}

In this section, we apply the established concentration inequalities to analyze the global optimality of the K-subspace clustering problem under the low-rank Gaussian mixture model (LRGMM).

\subsection{Problem Setup and Formulation}\label{subsec:MoG}

\begin{definition}[Low-Rank Gaussian Mixture Model]\label{def:MoG}
Let $C_1^\star,\dots,C_K^\star$ be a partition of $[N]$, where $N\ge K\ge 2$.
Let $n\ge d\ge 1$, and let $\bm U_k^\star\in\mathcal O^{n\times d}$ be an
orthonormal basis for each $k\in[K]$. We say that
$\{\bm z_i\}_{i=1}^N\subseteq\mathbb R^n$ is generated from the LRGMM if, for
each $k\in[K]$,
\begin{align}\label{eq:GMM}
\bm z_i=\bm U_k^\star \bm a_i+\bm e_i,
\qquad i\in C_k^\star,
\end{align}
where $\bm a_i\overset{\mathrm{i.i.d.}}{\sim}\mathcal N(\bm 0,\bm I_d)$ and
$\bm e_i\overset{\mathrm{i.i.d.}}{\sim}\mathcal N(\bm 0,\delta_k^2\bm I_n)$
for $i\in C_k^\star$. The latent vectors $\{\bm a_i\}_{i=1}^N$ and noises
$\{\bm e_i\}_{i=1}^N$ are mutually independent.
\end{definition}

Under this model, if $i\in C_k^\star$, then $\bm z_i$ lies near the subspace
spanned by $\bm U_k^\star$, with additive isotropic Gaussian noise. In particular,
\[
\bm z_i\sim
\mathcal N\left(
\bm 0,\,
\bm U_k^\star\bm U_k^{\star T}+\delta_k^2\bm I_n
\right),
\qquad i\in C_k^\star,
\]
where $\bm U_k^\star$ captures the underlying subspace and $\delta_k$ measures
the noise level in cluster $k$.

\vspace{-0.1in}
\paragraph{Problem formulation.}
Given unlabeled samples $\{\bm z_i\}_{i=1}^N$ and the number of clusters $K$,
the $K$-subspace clustering problem aims to recover both the subspace bases
$\{\bm U_k^\star\}_{k=1}^K$ and the cluster assignment
$\{C_k^\star\}_{k=1}^K$. We consider the following formulation:
\begin{align}\label{eq:loss}
 \min_{\bm U,C}\quad
F(\bm U,C)
:=
\sum_{k=1}^K\sum_{i\in C_k}
\left\|
\bm z_i-\bm U_k\bm U_k^T\bm z_i
\right\|^2
\quad
\st\quad
\bm U_k^T\bm U_k=\bm I_d,\ \forall k\in[K],
\end{align}
where $\bm U=\{\bm U_k\}_{k=1}^K$ denotes the estimated subspaces and
$C=\{C_k\}_{k=1}^K$ denotes a partition of $[N]$.

\begin{remark}\label{rem:dependecen}
Any global minimizer $(\hat{\bm U},\hat C)$ is a function of the full dataset
$\{\bm z_i\}_{i=1}^N$. Hence the selected subsets
$\{\bm z_i:i\in\hat C_k\}$ are data-dependent, which creates selection-induced
statistical dependence and makes the analysis of global optimality nontrivial.
\end{remark}

To evaluate recovery performance, we use permutation-invariant metrics. The
clustering error between an estimator $\hat C$ and the ground truth $C^\star$ is
\[
\textstyle  \ell(\hat C,C^\star)
:=
1-
\max_{\pi\in\Pi}
\frac{1}{N}
\sum_{k=1}^K
\left|
\hat C_{\pi(k)}\cap C_k^\star
\right|,
\]
where $\Pi$ is the set of all permutations. Similarly, for
$\hat{\bm U}:=\{\hat{\bm U}_k\}_{k=1}^K$, we define the subspace error
\[
\textstyle d_F(\hat{\bm U},\bm U^\star)
:=
\min_{\pi\in\Pi}
\sum_{k=1}^K
\left\|
\hat{\bm U}_{\pi(k)}\hat{\bm U}_{\pi(k)}^T
-
\bm U_k^\star\bm U_k^{\star T}
\right\|_F.
\]

\vspace{-0.1in}
\paragraph{Metrics for subspace clustering.}
The difficulty of subspace clustering is mainly governed by the separation
between subspaces, the noise level, and the cluster-size balance. We measure the
affinity between the $k$-th and $l$-th true subspaces by
\begin{align}\label{def:affi}
\mu_{kl}
:=
\|\bm U_k^{\star T}\bm U_l^\star\|_F
/ \sqrt d ,
\end{align}
and define $\mu_{\max}:=\max_{k\neq l}\mu_{kl},\,
\delta_{\max}:=\max_{k\in[K]}\delta_k.$
A larger subspace separation and a smaller noise level make clustering easier.
Accordingly, we define the signal-to-noise ratio as
\begin{align}\label{def:SNR}
\mathrm{SNR}
:=
(1-\mu_{\max}^2) / \delta_{\max}.
\end{align}
Finally, to quantify cluster-size imbalance, let $N_k^\star:=|C_k^\star|$ and
$N_{\min}^\star:=\min_{k\in[K]}N_k^\star$. We define
\begin{align}\label{def:kappa}
\kappa_N
:=
N / (N_{\min}^\star K).
\end{align}


\subsection{Global Optimality of Loss function}\label{subsec:opti}

In our analysis, we impose the following assumption on the total number of samples $N$, the number of samples in each subspace $\{N_k\}_{k=1}^K$, the subspace dimension $d$, and the ambient dimension $n$:
\begin{assumption}\label{AS:1}
The parameters $N$, $\{N_k^\star\}_{k=1}^K$, $d$ and $n$ in \Cref{def:MoG} satisfy
\begin{align}\label{eq:ndk}
   & N_{\min}^\star \gtrsim dK,\quad n \gtrsim \log n + \log N, \quad d \gtrsim \log d + \log N. 
\end{align}
\end{assumption}

It is worth mentioning that the condition \(N_{\min}^\star \gtrsim dK\) is mild and practically reasonable. In many clustering applications, although the ambient dimension \(n\) can be very large, the intrinsic or subspace dimension \(d\) is typically much smaller than the number of data points. 

Using this and \Cref{thm:lowerbound}, we can derive a recovery bound for the global optimal solutions of the K-subspace clustering problem \eqref{eq:loss} as follows: 

\begin{theorem}\label{thm:opti 1}
Consider the LRGMM in \Cref{def:MoG}, and suppose \Cref{AS:1} holds. Let $(\hat{\bm U},\hat{C})$ be an optimal solution of Problem \eqref{eq:loss}. Suppose that  
\begin{align}\label{eq:SNR1}
\mathrm{SNR} \gtrsim  \kappa_N K\sqrt{\frac{n}{d}} .
\end{align}
It holds with probability at least $1-N^{-\Omega(1)}$ that  
\begin{align}\label{eq:opti 1}
 \ell(\hat{C},C^\star) \lesssim \left( \kappa_N \sqrt{\frac{n}{d}}\cdot \frac{K} {\mathrm{SNR}} \right)^{\frac{1}{3}}. 
\end{align} 
\end{theorem}
The proof is deferred to \Cref{app:proof thm}. This theorem shows that the clustering error decays polynomially as the SNR ratio increases. In particular, higher SNR leads to more accurate recovery of the cluster assignments, and the error bound vanishes as the SNR grows sufficiently large.  

\vspace{-0.1in}
\paragraph{Experimental verification.} The recovery bound in \Cref{thm:opti 1} has universal constants. To empirically assess the magnitude of the constant, we simulate data from the LRGMM and compute
\[
R
:=
\frac{\ell(\hat C,C^\star)}
{\left(\sqrt{n/d}\,\kappa_N K/\mathrm{SNR}\right)^{1/3}},
\]
which is the ratio between the two sides of \eqref{eq:opti 1}. 
We vary the noise level \(\delta\), and hence the SNR, while fixing
\(K=2\), \(n=100\), \(N_1=N_2=500\), \(d=20\), and \(\mu_{\max}=0.67\). To compute \(\hat C\), we run the \(K\)-subspaces algorithm initialized by the thresholding inner-product spectral method proposed in \cite{wang2022convergence}. 
For each parameter setting, we generate independent data and repeat the experiment 10 times. 
The results are summarized in \Cref{tab:snr_error}. 
We observe that \(R\) remains bounded and exhibits small variance across the tested SNR values, suggesting that the unspecified constant in \Cref{thm:opti 1} is moderate in this setting. 

\begin{table}[t]
\centering
\caption{Empirical values of the ratio \(R\) under varying noise levels.}
\label{tab:snr_error}
\begin{tabular}{c c c c}
\toprule
\(\delta\) & SNR & Mean of \(R\) & Variance of \(R\) \\
\midrule
0.200 & 2.500 & 0.0000 & 0.00 \\
0.300 & 1.667 & 0.0007 & 0.00 \\
0.400 & 1.250 & 0.0034 & 0.00 \\
0.500 & 1.000 & 0.0129 & 0.00 \\
0.600 & 0.833 & 0.0397 & 0.00 \\
0.700 & 0.714 & 0.0862 & 0.01 \\
0.800 & 0.625 & 0.1858 & 0.03 \\
0.900 & 0.556 & 0.2247 & 0.02 \\
1.000 & 0.500 & 0.2226 & 0.01 \\
\bottomrule
\end{tabular} 
\end{table}

\subsection{Proof Outline}
In this subsection, we outline the proof of \Cref{thm:opti 1}. We first provide a necessary condition for the optimal solutions of Problem \eqref{eq:loss} as follows, whose proof is deferred to \Cref{proof:prop opti}. 

\begin{proposition}\label{prop:opti}
Let $(\hat{\bm U}, \hat{C})$ be an optimal solution of Problem \eqref{eq:loss}. It holds that for each $k \in [K]$,
\begin{align}\label{eq:Ck}
\hat{C}_k = \left\{ i \in [N]: \| \hat{\bm U}_k^T\bm z_i\| \ge \|\hat{\bm U}_l^T\bm z_i\|,\ \forall l \neq k \right\}.
\end{align}
Here, the columns of $\hat{\bm U}_k$ consist of eigenvectors of $\sum_{i \in \hat{C}_k} \bm z_i\bm z_i^T$ associated with top $d$ eigenvalues. 
\end{proposition}
Based on the optimality condition of $(\hat{\bm U},\hat{C})$, we have the following lemma, whose proof is deferred to \Cref{app:prop opt}. 
\begin{lemma}\label{lem:optimal}
    Let $(\hat{\bm U},\hat{C})$ be an optimal solution of Problem \eqref{eq:loss}. We have
\begin{equation}\label{ineq:sumE}
    \sum_{k=1}^K\sum_{l=1}^K \sum_{i \in \hat{C}_k\cap C_l^\star} \left(\|\bm a_i\|^2 - \|\hat{\bm U}_k^T\bm U_l^\star\bm a_i\|^2 \right) \le E,
\end{equation} 
where $E = 9N\left( \delta_{\max} \sqrt{nd} + \delta_{\max}^2 n \right)$.
\end{lemma}
An important observation is that
\begin{align}\label{ineq:lambdamin}
   \sum_{i \in \hat{C}_k\cap C_l^\star} \left(\|\bm a_i\|^2 - \|\hat{\bm U}_k^T\bm U_l^\star\bm a_i\|^2 \right) 
    = &  \left\langle \sum_{i \in \hat{C}_k\cap C_l^\star} \bm a_i \bm a_i^T, \bm I - \hat{\bm U}_k^T\bm U_l^\star(\hat{\bm U}_k^T\bm U_l^\star)^T \right\rangle \notag \\
  \geq & \lambda_{\min} \left( \sum_{i \in \hat{C}_k\cap C_l^\star} \bm a_i \bm a_i^T \right) \left( d - \|\hat{\bm U}_k^T\bm U_l^\star\|_F^2 \right).
\end{align}
It is important to note that any optimal solution $(\hat{\bm U}, \hat{C})$ of Problem \eqref{eq:loss} depends on random vectors $\{\bm a_i\}_{i=1}^N$ and $\{\bm e_i\}_{i=1}^N$; see \Cref{rem:dependecen}. Due to this dependence, we cannot directly apply the concentration inequality in \Cref{lem:cov esti} to estimate the singular values of $\sum_{i \in \hat{C}_k\cap C_l^\star} \bm a_i\bm a_i^T$. This challenge motivates the use of \Cref{thm:lowerbound}, which provides a way to circumvent the lack of independence induced by the data-dependent clustering. 

For any $k \in [K]$, if there exist some $l_1, l_2 \in [K]$ with $l_1 \neq l_2$ such that, for some $\kappa_1 \in (0, 1]$,
\begin{equation}\label{ineq:countercase}
    |\hat{C}_k \cap C_{l_1}^\star| \geq \kappa_1 \cdot N_{l_1}^\star,\quad |\hat{C}_k \cap C_{l_2}^\star| \geq \kappa_1 \cdot N_{l_2}^\star.
\end{equation}
By applying \Cref{thm:lowerbound}, we can get that, with probability given by \eqref{eq:prob},
\begin{equation}\label{ineq:theorem21apply}
   \lambda_{\min} \left( \sum_{i \in \hat{C}_k\cap C_{l}^\star} \bm a_i \bm a_i^T \right) \geq (1-\eta)^2\mu N_l^\star \quad \text{for } l = l_1 \text{ and } l_2,
\end{equation}
where $\eta \in (0, \tfrac{1}{2})$ and $\mu$ are defined by \eqref{eq:gamma mu}. Combining \eqref{ineq:sumE}, \eqref{ineq:lambdamin}, and \eqref{ineq:theorem21apply} gives
$$
d - \|\hat{\bm U}_k^T\bm U_{l_1}^\star\|_F^2  +  d - \|\hat{\bm U}_k^T\bm U_{l_2}^\star\|_F^2 \leq \frac{E}{(1-\eta)^2\mu N_{\min}^\star}.
$$
A simple computation shows that
\begin{align}\label{ineq:triangle}
    d - \|\bm U_{l_1}^{\star T}\bm U_{l_2}^\star\|_F^2 & \leq 4 \left( d - \|\hat{\bm U}_k^T\bm U_{l_1}^\star\|_F^2  +  d - \|\hat{\bm U}_k^T\bm U_{l_2}^\star\|_F^2 \right) \leq \frac{4E}{(1-\eta)^2\mu N_{\min}^\star}.
\end{align}
Since $d - \|\bm U_{l_1}^{\star T}\bm U_{l_2}^\star\|_F^2 \geq (1-\mu_{\max}^2)d$ is lower bounded, the scenario as stated in \eqref{ineq:countercase} will not happen as long as the noise level and thus $E$ is small enough. In particular, we have the following proposition to bound the number of misclassification points, whose proof can be found in \Cref{app:prop error rate}.

\begin{proposition}\label{prop:error rate}
Let $(\hat{\bm U},\hat{C})$ be an optimal solution of Problem \eqref{eq:loss}. Suppose that  
\begin{align}\label{eq:delta 1}
\delta_{\max} \le \frac{\mu (1-\eta)^2 (1-\mu_{\max}^2)}{40 \kappa_N K } \sqrt{\frac{d}{n}},
\end{align}
for some $\eta \in (0, 1/2)$, $\kappa_1 \in (0, 1/K)$, and $\mu$ is defined by \eqref{eq:gamma mu}.
It holds with probability at least 
\begin{align*}
     1 - 2 K\left( \frac{22}{\eta\sqrt{\mu }} + 1 \right)^d \left( \exp\left( -\frac{\eta^2 \kappa_1 N^\star_{\min}}{2} \right)  + \exp\left( - c_1\eta^2\mu^2 N^\star_{\min} \right) \right) - 2K\exp(-2N^\star_{\min}),
\end{align*}
that there exists a unique permutation $\pi:[K] \to [K]$ such that
\begin{align}\label{eq:rate 1}
|\hat{C}_{\pi(k)} \cap C_k^\star| \ge \left(1 - (K-1) \kappa_1 \right) N_k^\star,\ \forall k \in [K]. 
\end{align} 
\end{proposition}

Finally, the proof of \Cref{thm:opti 1} relies on the fact that, given any $\kappa_1 \in [0,1]$, we always have  
\begin{equation}\label{ineq:kappamu3}
   \mu \gtrsim \kappa_1^3. 
\end{equation}
A detailed proof can be found in \Cref{app:proof thm}.

\section{Conclusions and Limitations}\label{sec:con}

In this paper, we studied spectral concentration for sample covariance matrices formed from arbitrary, possibly data-dependent subsets of random vectors. For i.i.d. Gaussian vectors, we established sharp upper and lower eigenvalue bounds and justified their tightness through complementary estimates. We further extended the framework to sub-Gaussian vectors and weakly dependent observations, with geometrically strong-mixing Gaussian sequences as a representative example. As an application, we used our concentration inequalities to analyze the $K$-subspace clustering problem and showed that the clustering error of global minimizers decreases polynomially with the signal-to-noise ratio. The main concentration inequalities are derived under Gaussian or sub-Gaussian assumptions. Although these distributions cover many important settings, extending the results to heavier-tailed random vectors would require additional truncation or robustification techniques.

\section*{Acknowledgement}

H. Liu is supported in part by the National Natural Science Foundation of China (Grant NSFC-12301403, 72192830, 72192832). P. Wang is supported in part by the University of Macau under grants SRG2025-00043-FST and UMDF-TISF-I/2026/013/FST, and in part by the Macau Science and Technology Development Fund (FDCT) 0091/2025/ITP2. L. Balzano is supported in part by NSF award 2331590, the Charles Simonyi Endowment, and the University of Michigan Crosby award.


\bibliographystyle{abbrv}
\bibliography{mixture_model,icml2026/reference}

\newpage
\appendix
\onecolumn
\begin{center}
{\Large \bf Appendices}
\end{center}\vspace{-0.15in}
\par\noindent\rule{\textwidth}{1pt}

\section{Proofs for \Cref{sec:main}}

\subsection{Proof of \Cref{thm:lowerbound}} \label{sec:proof-lowerbound}
\begin{proof}
For ease of exposition, let $\bm A_{\Omega} \in \R^{m \times |\Omega|}$ be a matrix whose columns consist of $\bm a_i$ for each $i \in \Omega$ and $\bm A \in \R^{m \times M} $ be a matrix whose columns consist of $\bm a_i$ for all $i \in [M]$. Applying \Cref{lem:cov esti} to $\bm A\bm A^T$ with $\eta=\sqrt{M}$ yields that it holds with probability at least $1-2\exp(-2M)$ that
\begin{align}
\left\|\frac{1}{M} \bm A\bm A^T - \bm I_m \right\| \le \frac{9\left(\sqrt{m}+\sqrt{M}\right)}{\sqrt{M}} \; \le 18, \label{eq:aainormbound}
\end{align}
where the last inequality holds since we assume $M \geq m$. We then have 
\begin{align}\label{ineq:A-norm}
    \frac{1}{M} \|\bm A\|^2 = \left\|\frac{1}{M} \bm A\bm A^T\right\| \leq \left\|\frac{1}{M} \bm A\bm A^T - \bm I_m\right\| + \|\bm I_m\| \leq 19.
\end{align}
Let $\mathcal{N}_{\varepsilon} \subseteq \mathbb{S}^{m-1}$ be an $\varepsilon$-net of $\mathbb{S}^{m-1}$, i.e., for any $\bm x \in \mathbb{S}^{m-1}$, there exists $\bm x_0 \in \mathcal{N}_{\varepsilon}$ such that $\|\bm x - \bm x_0\| \le \varepsilon$. Using \citep[Corollary 4.2.13]{vershynin2018high}, the cardinality of the $\varepsilon$-net $\mathcal{N}_{\varepsilon}$ is
\begin{align}\label{eq0:lem cov A1}
\left|\mathcal{N}_{\varepsilon} \right| \le \left( \frac{2}{\varepsilon} + 1 \right)^m. 
\end{align} 
Let $\bm x \in \mathbb{S}^{m-1}$ be such that $\sigma_{\min}\left( \bm A_{\Omega}  \right) = \|\bm A_{\Omega}^T \bm x\|$ and $\bm x_0 \in \mathcal{N}_{\varepsilon}$ be such that $\|\bm x - \bm x_0\| \le \varepsilon$. Using the triangle inequality, we obtain
\begin{align}\label{eq5:lem cov A1}
\sigma_{\min}\left( \bm A_{\Omega}  \right) \ge \|\bm A_{\Omega}^T \bm x_0\| - \varepsilon\|\bm A_{\Omega}\| \ge \min_{\bm x \in \mathcal{N}_{\varepsilon}} \|\bm A_{\Omega}^T\bm x\| - \varepsilon\|\bm A_{\Omega}\| \ge \min_{\bm x \in \mathcal{N}_{\varepsilon}} \|\bm A_{\Omega}^T\bm x\| - \varepsilon\|\bm A\|. 
\end{align} 
We claim and prove below that it holds with probability defined below in \eqref{eq:prop spec 1} that 
\begin{align}\label{eq3:lem cov A1}  
& \min_{\bm x \in \mathcal{N}_{\varepsilon}} \|\bm A_{\Omega}^T\bm x\|^2 \geq (1 -\eta) \mu M \;,
\end{align}
where $\mu$ is defined in \eqref{eq:gamma mu}. Substituting this and \eqref{ineq:A-norm} into \eqref{eq5:lem cov A1} yields 
\begin{align*}
\sigma_{\min}\left( \bm A_{\Omega} \right) \ge \sqrt{(1 -\eta) \mu M} - \varepsilon \sqrt{19 M}.
\end{align*}
From here to get \eqref{eq:spec min}, we let $\varepsilon = \sqrt{\mu} \eta/ 11$, which gives us
\begin{align*}
        \sigma_{\min}\left( \bm A_{\Omega} \right) \ge \sqrt{\mu M}\left(\sqrt{1 -\eta} - \frac{\sqrt{19}}{11} \eta\right) \ge (1 -\eta) \sqrt{\mu M}
\end{align*}
where the last inequality holds because $\eta \leq 1/2$ implies $$
\sqrt{1 -\eta} - (1 -\eta) = \frac{(1 -\eta) - (1 -\eta)^2}{\sqrt{1 -\eta} + (1 -\eta)} = \eta \cdot \frac{1}{1 + 1 /\sqrt{1 -\eta}} \geq \frac{1}{1 +\sqrt{2}} \eta \geq  \frac{\sqrt{19}}{11} \eta.
$$
Using the union bound, we obtain the probability of the event. 

The rest of the proof is devoted to proving the claim \eqref{eq3:lem cov A1}. We fix $\bm x \in \mathcal{N}_{\epsilon}$ and let  
\begin{align}\label{eq2:lem cov A1 1}
\mathcal{A} := \left\{ i  \in [M]: \left(\bm a_i^T\bm x\right)^2 \le \gamma^2  \right\},\ \text{where}\ \gamma := \Phi^{-1}\left( \frac{1 + (1 - \eta)\kappa_1}{2} \right).
\end{align}
According to $\bm a_i  \overset{i.i.d.}{\sim} \mathcal{N}(\bm 0, \bm I)$ for each $i \in [M]$ and $\bm x \in \mathcal{N}_{\varepsilon}$, we have $\bm a_i^T\bm x \sim \mathcal{N}(0,1)$ for each $i \in [M]$. Therefore, we obtain 
\begin{align}\label{eq4:lem cov A1 1}
\P\left( \left(\bm a_i^T\bm x\right)^2 \le \gamma^2 \right) = \P\left( -\gamma \le \bm a_i^T\bm x \le \gamma \right) = 2\Phi\left( \gamma\right) - 1 = (1 - \eta)\kappa_1.
\end{align}
It follows from \eqref{eq2:lem cov A1 1} that  $|\mathcal{A}| = \sum_{i =1}^M \mathbb{I}\left\{ \left(\bm a_i^T\bm x\right)^2 \le \gamma^2  \right\}$, 
where $\mathbb{I}\left\{ \left(\bm a_i^T\bm x\right)^2 \le \gamma^2  \right\} = 1$ if $\left(\bm a_i^T\bm x\right)^2 \le \gamma^2$ and $0$ otherwise. This, together with \eqref{eq4:lem cov A1 1}, implies 
\begin{align*}
\E\left[ |\mathcal{A}| \right] = (1 - \eta)\kappa_1 M,\quad \mathrm{Var}\left( |\mathcal{A}| \right) \le (1 - \eta)\kappa_1 M.
\end{align*}
Applying Bernstein's inequality (see, e.g., \citep[Theorem 2.8.4]{vershynin2018high}) to $|\mathcal{A}|$, which is the sum of i.i.d. Bernoulli random variables, yields that 
\begin{align}\label{ineq:bernstein bernoulli}
\P\left( \Big| |\mathcal{A}| - (1 - \eta)\kappa_1 M\Big| \ge \eta\kappa_1 M \right) & \le 2\exp\left(-\frac{(\eta\kappa_1 M)^2}{2\left((1 - \eta)\kappa_1 M + \eta\kappa_1 M/{3} \right)}\right)
\le 2\exp\left( -\frac{ \eta^2\kappa_1 M}{2} \right). 
\end{align}
This implies that it holds with probability at least $1-2\exp\left( -{\eta^2\kappa_1 M}/{2} \right)$ that $|\mathcal{A}| \le \kappa_1 M$. Since $|\Omega| \ge \kappa_1 M$, we obtain 
\begin{align}\label{eq:boundsmallerset 1}
\|\bm A_{\Omega}^T\bm x\|^2 = \sum_{i \in \Omega} \left(\bm a_i^T\bm x\right)^2 \ge \sum_{i \in \mathcal{A}} \left(\bm a_i^T\bm x\right)^2 = \sum_{ i=1 }^M  \left(\bm a_i^T\bm x\right)^2 \mathbb{I}\left\{  \left(\bm a_i^T\bm x\right)^2 \le \gamma^2 \right\}, 
\end{align}
where the last equality follows from \eqref{eq2:lem cov A1 1}. Since $\bm a_i^T\bm x \overset{i.i.d.}{\sim} \mathcal{N}(0,1)$ for each $i \in [M]$, we compute
\begin{align}\label{eq6:lem cov A1 1}
\mu & := \E\left[ \left(\bm a_i^T\bm x\right)^2 \mathbb{I}\left\{  \left(\bm a_i^T\bm x\right)^2 \le \gamma^2 \right\} \right] = \sqrt{\frac{2}{\pi}}\int_0^{\gamma} x^2 \exp\left( -\frac{x^2}{2} \right) dx \notag \\
& = -\sqrt{\frac{2}{\pi}} \left( x\exp\left( -\frac{x^2}{2} \right) \bigg|_0^{\gamma} - \int_0^{\gamma} \exp\left( -\frac{x^2}{2} \right) dx \right) = (1-\eta)\kappa_1 - \sqrt{\frac{2\gamma^2}{\pi}}\exp\left( -\frac{\gamma^2}{2} \right),
\end{align}
where the last equality uses the definition of $\gamma$ in \eqref{eq2:lem cov A1 1}. 
Since $\bm a_i^T\bm x$ is Gaussian, $\left(\bm a_i^T\bm x\right)^2$ is subexponential. Applying Hoeffding's inequality for bounded distribution   \citep[Theorem 2.2.6]{vershynin2018high} yields
\begin{align}\label{ineq:subexponential-lowerbound bounded}
\P\left( \left|\sum_{i=1}^M \left(\bm a_i^T\bm x\right)^2 \mathbb{I}\left\{  \left(\bm a_i^T\bm x\right)^2 \le \gamma^2 \right\} - M \mu \right| \ge \eta M \mu \right) \le 2\exp\left( - \frac{2\eta^2\mu^2 M}{\gamma^4} \right). 
\end{align}
However, the above bound is invalid when $\gamma$ goes to infinity, so we have an alternative bound. Applying Subexponential Bernstein inequality \citep[Corollary 2.8.3]{vershynin2018high} yields
\begin{align}\label{ineq:subexponential-lowerbound}
\P\left( \left|\sum_{i=1}^M \left(\bm a_i^T\bm x\right)^2 \mathbb{I}\left\{  \left(\bm a_i^T\bm x\right)^2 \le \gamma^2 \right\} - M \mu \right| \ge \eta M \mu \right) \le 2\exp\left( - \frac{c' \cdot\eta^2\mu^2 M}{K^2} \right), 
\end{align}
where $c'>0$ is an absolute constant and $K = \left\| \bm z^2 \right\|_{\psi_1} \geq \left\| \bm z^2 \mathbb{I}\left\{  |\bm z| \le \gamma \right\} \right\|_{\psi_1}$ for $z \sim \mathcal{N}(0,1)$. Let $c = c' / K^2$, which is also an absolute constant.
This implies that it holds with probability at least $1-2\exp\left( - { c\cdot\eta^2\mu^2 M} \right)$ that 
\begin{align*}
 \sum_{i = 1}^M \left(\bm a_i^T\bm x\right)^2 \mathbb{I}\left\{  \left(\bm a_i^T\bm x\right)^2 \le \gamma^2 \right\}  \ge  (1 - \eta) M \mu, 
\end{align*}
This, together with 
\eqref{eq:boundsmallerset 1} 
implies \eqref{eq3:lem cov A1}. 
Finally, using the union bound and \eqref{eq0:lem cov A1} yields that \eqref{eq3:lem cov A1} holds with probability at least 
\begin{align}\label{eq:prop spec 1}
& 1 - 2 \left( \frac{22}{\eta\sqrt{\mu }} + 1 \right)^m \left( \exp\left( -\frac{\eta^2 \kappa_1 M}{2} \right)  + \min\left\{ \exp\left( - c \cdot\eta^2\mu^2 M \right), \exp\left( - \frac{2\eta^2\mu^2 M}{\gamma^4} \right) \right\} \right).
\end{align}
\end{proof}

\subsection{Lower Bound of $\kappa_1$ in \Cref{thm:lowerbound}}\label{app:kappa}

Formally, \Cref{thm:lowerbound} is stated for any $\kappa_1 \in (0,1]$. What depends on $\kappa_1$ is whether the probability bound in \eqref{eq:prob} is non-vacuous and quantitatively strong. We write the probability in \eqref{eq:prob} schematically as
\[
1-2\left(1+\frac{22}{\eta\sqrt{\mu}}\right)^m
\Bigl(\exp\left({-c\,\eta^2\kappa_1 M}\right) + \exp\left({-c\,\eta^2\mu^2 M}\right)\Bigr) - 2\exp\left({-2M}\right),
\]
A convenient sufficient condition ensuring that the above probability is positive is
\[
m\log\!\Bigl(1+\frac{22}{\eta\sqrt{\mu}}\Bigr)
\ll
\eta^2\kappa_1 M,
\quad
m\log\!\Bigl(1+\frac{22}{\eta\sqrt{\mu}}\Bigr)
\ll
\eta^2\mu^2 M.
\]
where $\mu=\mu(\kappa_1,\eta)$ is defined in \Cref{thm:lowerbound}. Thus, the valid regime for $\kappa_1$ is determined by the requirement that these exponents dominate the covering-number term. Using
\[
\gamma=\Phi^{-1}\!\left(\frac{1+(1-\eta)\kappa_1}{2}\right)=\Theta(\kappa_1),
\quad
\mu=\sqrt{\frac{2}{\pi}}\int_0^\gamma x^2 e^{-x^2/2}\,dx=\Theta(\kappa_1^3),
\]
we obtain the rough scaling
\[
\frac{M}{m}\gtrsim \kappa_1^{-6}\log\frac{1}{\kappa_1},
\]
up to constants and the dependence on $\eta$. Equivalently, for a fixed ratio $M/m$, the smallest admissible $\kappa_1$ scales roughly like
\[
\kappa_1 \gtrsim \Bigl(\frac{m\log(M/m)}{M}\Bigr)^{1/6}.
\]
This explains why \Cref{exam:parameter} shows an apparent lower bound on $\kappa_1$: those numbers are not theorem-level hard thresholds, but illustrative choices ensuring that the probability bound is already numerically strong (e.g., at least $0.999$). 

\subsection{Proof of \Cref{thm:upperbound}} \label{sec:proof-upperbound} 

\begin{proof}
Using the $\varepsilon$-net and \citep[Exercise 4.4.3]{vershynin2018high}, we have 
\begin{align}\label{eq1:prop cov A}
     \left\| \sum_{i \in \Omega} \bm a_i\bm a_i^T  \right\| \le \frac{1}{1-2\varepsilon} \max_{\bm x \in \mathcal{N}_{\varepsilon}}  \sum_{i \in \Omega} \left( \bm a_i^T\bm x\right)^2. 
\end{align}
We fix $\bm x \in \mathcal{N}_{\varepsilon}$ and let 
\begin{align}\label{eq2:prop cov A 3}
\mathcal{A} := \left\{ i  \in [M]: \left(\bm a_i^T\bm x\right)^2 \ge \gamma^2  \right\},\ \text{where}\ \gamma := \Phi^{-1}\left( 1 - \frac{(1+\eta)\kappa_2}{2} \right).
\end{align} 
According to $\bm a_i \overset{i.i.d.}{\sim} \mathcal{N}(\bm 0, \bm I)$ and $\bm x \in \mathcal{N}_{\varepsilon}$, we have $\bm a_i^T\bm x  \overset{i.i.d.}{\sim} \mathcal{N}(0,1)$ for each $i \in [M]$. Therefore, we obtain for each $i \in [M]$,
\begin{align}\label{eq3:prop cov A 3}
\P\left( \left(\bm a_i^T\bm x\right)^2 \ge \gamma^2 \right) = 2\P\left( \bm a_i^T\bm x \ge \gamma \right) = 2\left( 1 - \Phi(\gamma) \right) = (1 + \eta)\kappa_2.
\end{align}  
It follows from \eqref{eq2:prop cov A 3} that  $|\mathcal{A}| = \sum_{i =1}^M \mathbb{I}\left\{ \left(\bm a_i^T\bm x\right)^2 \ge \gamma^2  \right\}$, where $\mathbb{I}\left\{ \left(\bm a_i^T\bm x\right)^2 \ge \gamma^2  \right\} = 1$ if $\left(\bm a_i^T\bm x\right)^2 \ge \gamma^2$ and $0$ otherwise. This, together with \eqref{eq3:prop cov A 3}, implies 
\begin{align}\label{eq4:prop cov A 3}
    \E\left[ |\mathcal{A}| \right] = (1 + \eta)\kappa_2 M,\quad \mathrm{Var}\left( |\mathcal{A}| \right) \le (1 + \eta)\kappa_2 M.  
\end{align}
Applying Bernstein's inequality (see, e.g., \citep[Theorem 2.8.4]{vershynin2018high}) to $|\mathcal{A}|$, which is the sum of i.i.d. Bernoulli random variables, yields 
\begin{align*}
\P\left( \left| |\mathcal{A}| - (1 +\eta)\kappa_2 M \right| \ge \eta\kappa_2 M \right) & \le 2\exp\left(-\frac{\eta^2 \kappa_2^2 M^2}{2\left( (1+\eta)\kappa_2 M + \eta \kappa_2 M/{3} \right)}\right) \le 2 \exp \left(-\frac{\eta^2 \kappa_2 M}{4}\right),
\end{align*} 
where the last inequality holds because $\eta \le 1/2$. This implies that it holds with probability at least $1-2\exp\left(-\eta^2\kappa_2M/4 \right)$ that $|\mathcal{A}| \ge \kappa_2 M$.  Since $|\Omega| \le \kappa_2 M$, we obtain 
\begin{align}\label{eq5:prop cov A 3}
\|\bm A_{\Omega}^T\bm x\|^2 = \sum_{i \in \Omega} \left(\bm a_i^T\bm x\right)^2 \le \sum_{i \in \mathcal{A}} \left(\bm a_i^T\bm x\right)^2 = \sum_{ i=1 }^M  \left(\bm a_i^T\bm x\right)^2 \mathbb{I}\left\{  \left(\bm a_i^T\bm x\right)^2 \ge \gamma^2 \right\}, 
\end{align}
where the last equality follows from \eqref{eq2:prop cov A 3}. Since $\bm a_i^T\bm x \overset{i.i.d.}{\sim} \mathcal{N}(0,1)$ for each $i \in [M]$, we compute
\begin{align}\label{eq6:prop cov A 3} 
\mu & := \E\left[ \left(\bm a_i^T\bm x\right)^2 \mathbb{I}\left\{ \left(\bm a_i^T\bm x\right)^2 \ge \gamma^2 \right\} \right] = \sqrt{\frac{2}{\pi}}\int_{\gamma }^\infty x^2 \exp\left( -\frac{x^2}{2} \right) dx \notag \\ 
& = -\sqrt{\frac{2}{\pi}} \left( x\exp\left( -\frac{x^2}{2} \right) \bigg|_{\gamma}^\infty - \int_{\gamma}^\infty \exp\left( -\frac{x^2}{2} \right) dx \right) = (1+\eta)\kappa_2 + \sqrt{\frac{2\gamma^2}{\pi}}\exp\left( -\frac{\gamma^2}{2} \right), 
\end{align}
where the last equality uses the definition of $\gamma$ in \eqref{eq2:prop cov A 3}. Similar to \eqref{ineq:subexponential-lowerbound}, Subexponential Bernstein inequality \citep[Corollary 2.9.2]{vershynin2018high} yields
\begin{align}\label{ineq:subexponential-lowerbound 3}
\P\left( \left|\sum_{i=1}^M \left(\bm a_i^T\bm x\right)^2 \mathbb{I}\left\{  \left(\bm a_i^T\bm x\right)^2 \ge \gamma^2 \right\} - M \mu \right| \ge \eta M \mu \right) \le 2\exp\left( - c_2 \cdot\eta^2\mu^2 M \right), 
\end{align}
for some absolute constant $c_2 > 0$.
This implies that, with probability at least $1 - 2\exp\left( - c_2 \cdot\eta^2\mu^2 M \right)$, 
\begin{align*}
    \sum_{i = 1}^M \left(\bm a_i^T\bm x\right)^2 \mathbb{I}\left\{ \left(  \bm a_i^T\bm x\right)^2 \ge \gamma^2 \right\} \le (1+\eta) \mu M. 
\end{align*}
This, together with \eqref{eq1:prop cov A} and \eqref{eq5:prop cov A 3}, implies 
$$
\left\| \sum_{i \in \Omega} \bm a_i\bm a_i^T  \right\| \le \frac{(1+\eta) \mu M}{1-2\varepsilon}. 
$$ 
Setting $\varepsilon = \eta/3$, since $\eta \in (0, 1/2)$, we have $\frac{1}{1-2\varepsilon} \leq 1 +\eta$, which together with the above inequality, yields \eqref{eq:spec max}. Using \eqref{eq0:lem cov A1} and the union bound yields that \eqref{eq:spec max} holds with probability at least
\begin{align*}
    1 - 2\left( \frac{6}{\eta} + 1 \right)^m\left( \exp \left(-\frac{\eta^2\kappa_2 M}{6}\right)  + \exp\left(- c_2 \cdot\eta^2\mu^2 M \right) \right). 
\end{align*}
\end{proof}

\subsection{Proof of \Cref{coro:lowerboundtightness}}\label{app subsec:coro}

\paragraph{Proof of Part (i) in \Cref{coro:lowerboundtightness}.} We fix some $\bm x_0 \in \mathbb{S}^{m-1}$ and let  
\begin{align}\label{eq2:lem cov A1 2}
\widetilde{\mathcal{A}} := \left\{ i  \in [M]: \left(\bm a_i^T\bm x_0 \right)^2 \le \widetilde{\gamma}^2  \right\},\ \text{where}\ \widetilde{\gamma} := \Phi^{-1}\left( \frac{1 + (1 + \eta)\kappa_1}{2} \right).
\end{align}
Similar to \eqref{eq4:lem cov A1 1}, we know that \begin{align}\label{eq4:lem cov A1 tilde}
\P\left( \left(\bm a_i^T\bm x_0\right)^2 \le \widetilde{\gamma}^2 \right) = \P\left( -\widetilde{\gamma} \le \bm a_i^T\bm x_0 \le \widetilde{\gamma} \right) = 2\Phi\left( \widetilde{\gamma}\right) - 1 = (1 + \eta)\kappa_1,
\end{align}
which implies that $\E\left[ |\widetilde{\mathcal{A}}| \right] = (1 + \eta)\kappa_1 M$ and $\mathrm{Var}\left( |\widetilde{\mathcal{A}}| \right) \le (1 + \eta)\kappa_1 M.$ Then, similar to \eqref{ineq:bernstein bernoulli}, we have
\begin{align}\label{ineq:bernstein bernoulli tilde}
\P\left( \Big| |\widetilde{\mathcal{A}}| - (1 + \eta)\kappa_1 M\Big| \ge \eta\kappa_1 M \right) & \le 2\exp\left(-\frac{(\eta\kappa_1 M)^2}{2\left((1 + \eta)\kappa_1 M + \eta\kappa_1 M/{3} \right)}\right)
\le 2\exp\left( -\frac{ \eta^2\kappa_1 M}{4} \right). 
\end{align}
This implies that it holds with probability at least $1-2\exp\left( -{\eta^2\kappa_1 M}/{4} \right)$ that $|\widetilde{\mathcal{A}}| \ge \kappa_1 M$. By choosing $\Omega = \widetilde{\mathcal{A}}$, we obtain 
\begin{align}\label{eq:boundsmallerset tilde}
\lambda_{\min}\left( \sum_{i \in \Omega} \bm a_i\bm a_i^T\right) \leq \|\bm A_{\Omega}^T\bm x_0\|^2 = \sum_{i \in \widetilde{\mathcal{A}}} \left(\bm a_i^T\bm x_0\right)^2 = \sum_{ i=1 }^M  \left(\bm a_i^T\bm x_0\right)^2 \mathbb{I}\left\{  \left(\bm a_i^T\bm x_0\right)^2 \le \widetilde{\gamma}^2 \right\}, 
\end{align}
where the last equality follows from \eqref{eq2:lem cov A1 2}. Since $\bm a_i^T\bm x_0 \overset{i.i.d.}{\sim} \mathcal{N}(0,1)$ for each $i \in [M]$, we compute
\begin{align}\label{eq6:lem cov A1}
\widetilde{\mu} & := \E\left[ \left(\bm a_i^T\bm x_0\right)^2 \mathbb{I}\left\{  \left(\bm a_i^T\bm x_0\right)^2 \le \widetilde{\gamma}^2 \right\} \right] = \sqrt{\frac{2}{\pi}}\int_0^{\widetilde{\gamma}} x^2 \exp\left( -\frac{x^2}{2} \right) dx \notag \\
& = -\sqrt{\frac{2}{\pi}} \left( x\exp\left( -\frac{x^2}{2} \right) \bigg|_0^{\widetilde{\gamma}} - \int_0^{\widetilde{\gamma}} \exp\left( -\frac{x^2}{2} \right) dx \right) = (1+\eta)\kappa_1 - \sqrt{\frac{2\widetilde{\gamma}^2}{\pi}}\exp\left( -\frac{\widetilde{\gamma}^2}{2} \right),
\end{align}
where the last equality uses the definition of $\widetilde{\gamma}$ in \eqref{eq2:lem cov A1 2}. The same concentration inequality as \eqref{ineq:subexponential-lowerbound} implies that it holds with probability at least $1-2\exp\left( - { c\cdot\eta^2 \widetilde{\mu}^2 M} \right)$ that 
\begin{align*}
 \sum_{i = 1}^M \left(\bm a_i^T\bm x_0 \right)^2 \mathbb{I}\left\{  \left(\bm a_i^T\bm x_0 \right)^2 \le \widetilde{\gamma}^2 \right\}  \le  (1 + \eta) M \widetilde{\mu}, 
\end{align*}
Thus, the result holds with probability at least 
$$1- 2\exp\left( -\frac{ \eta^2\kappa_1 M}{4} \right) - 2\exp\left( - { c\cdot\eta^2 \widetilde{\mu}^2 M} \right).$$

\paragraph{Proof of Part (ii) in \Cref{coro:lowerboundtightness}.}

For any fixed $\bm x_0 \in \mathbb{S}^{m-1}$, let 
\begin{align}\label{eq2:prop cov A}
\widetilde{\mathcal{A}} := \left\{ i  \in [M]: \left(\bm a_i^T\bm x_0 \right)^2 \ge \widetilde{\gamma}^2  \right\},\ \text{where}\ \widetilde{\gamma} := \Phi^{-1}\left( 1 - \frac{(1 - \eta)\kappa_2}{2} \right).
\end{align} 
According to $\bm a_i \overset{i.i.d.}{\sim} \mathcal{N}(\bm 0, \bm I)$ and $\bm x_0 \in \mathbb{S}^{m-1}$, we have $\bm a_i^T\bm x_0  \overset{i.i.d.}{\sim} \mathcal{N}(0,1)$ for each $i \in [M]$. Therefore, we obtain for each $i \in [M]$,
\begin{align}\label{eq3:prop cov A}
\P\left( \left(\bm a_i^T\bm x_0\right)^2 \ge \widetilde{\gamma}^2 \right) = 2\P\left( \bm a_i^T\bm x_0 \ge \widetilde{\gamma}\right) = 2\left( 1 - \Phi(\widetilde{\gamma}) \right) = (1 - \eta)\kappa_2.
\end{align}  
It follows from \eqref{eq2:prop cov A} that  $|\widetilde{\mathcal{A}}| = \sum_{i =1}^M \mathbb{I}\left\{ \left(\bm a_i^T\bm x_0\right)^2 \ge \widetilde{\gamma}^2  \right\}$, where $\mathbb{I}\left\{ \left(\bm a_i^T\bm x_0\right)^2 \ge \widetilde{\gamma}^2  \right\} = 1$ if $\left(\bm a_i^T\bm x_0\right)^2 \ge \widetilde{\gamma}^2$ and $0$ otherwise. This, together with \eqref{eq3:prop cov A}, implies 
\begin{align}\label{eq4:prop cov A}
    \E\left[ |\widetilde{\mathcal{A}}| \right] = (1 - \eta)\kappa_2 M,\quad \mathrm{Var}\left( |\widetilde{\mathcal{A}}| \right) \le (1 - \eta)\kappa_2 M.  
\end{align}
Applying Bernstein's inequality (see, e.g., \citep[Theorem 2.8.4]{vershynin2018high}) to $|\widetilde{\mathcal{A}}|$, which is the sum of i.i.d. Bernoulli random variables, yields 
\begin{align*}
\P\left( \left| |\widetilde{\mathcal{A}}| - (1 -\eta)\kappa_2 M \right| \ge \eta\kappa_2 M \right) & \le 2\exp\left(-\frac{\eta^2 \kappa_2^2 M^2}{2\left( (1-\eta)\kappa_2 M + \eta \kappa_2 M/{3} \right)}\right) \le 2 \exp \left(-\frac{\eta^2 \kappa_2 M}{2}\right). 
\end{align*} 
This implies that it holds with probability at least $1-2\exp\left(-\eta^2\kappa_2M/2 \right)$ that $|\widetilde{\mathcal{A}}| \le \kappa_2 M$. By choosing $\Omega = \widetilde{\mathcal{A}}$, we obtain 
\begin{align}\label{eq5:prop cov A 4}
\left\| \sum_{i \in \Omega} \bm a_i\bm a_i^T\right\| \ge \|\bm A_{\Omega}^T\bm x_0\|^2 = \sum_{i \in \widetilde{\mathcal{A}}} \left(\bm a_i^T\bm x_0\right)^2 = \sum_{ i=1 }^M  \left(\bm a_i^T\bm x_0\right)^2 \mathbb{I}\left\{  \left(\bm a_i^T\bm x_0\right)^2 \ge \widetilde{\gamma}^2 \right\},
\end{align}
Just similar to \eqref{eq6:prop cov A 3} and \eqref{ineq:subexponential-lowerbound 3}, we have
\begin{align}\label{eq6:prop cov A 4} 
\widetilde{\mu} & := \E\left[ \left(\bm a_i^T\bm x_0\right)^2 \mathbb{I}\left\{ \left(\bm a_i^T\bm x_0\right)^2 \ge \widetilde{\gamma}^2 \right\} \right]  = (1 - \eta)\kappa_2 + \sqrt{\frac{2\widetilde{\gamma}^2}{\pi}}\exp\left( -\frac{\widetilde{\gamma}^2}{2} \right). 
\end{align}
and 
\begin{align}\label{ineq:subexponential-lowerbound 4}
\P\left( \left|\sum_{i=1}^M \left(\bm a_i^T\bm x_0\right)^2 \mathbb{I}\left\{  \left(\bm a_i^T\bm x_0\right)^2 \ge \widetilde{\gamma}^2 \right\} - M \widetilde{\mu} \right| \ge \eta M \widetilde{\mu} \right) \le 2\exp\left( - c_2 \cdot\eta^2\widetilde{\mu}^2 M \right), 
\end{align}
Thus, the result holds with probability at least 
$$1- 2\exp\left( -\frac{ \eta^2\kappa_2 M}{2} \right) - 2\exp\left( - { c_2\cdot\eta^2 \widetilde{\mu}^2 M} \right).$$

\subsection{Weakness of Simple Union Bound}\label{app:union}

To address the difficulty caused by the dependence induced by $\Omega$, a natural idea is to use a union bound over all possible choices of $\Omega$ based on \Cref{lem:cov esti}.  
However, this bound is only valid when $\kappa_1 > \exp\left({-{1}/{4}}\right) \approx 0.78$. That is, $\kappa_1$ needs to be very close to 1. 

Now, we provide evidence to support the above claim. For any fixed set $\Omega$ (independent of $\{\bm a_i\}_{i=1}^M$) with cardinality $|\Omega| = \kappa_1 M$, it follows from \Cref{lem:cov esti} that 
\begin{align}\label{eq1:cov}
\left\| \frac{1}{\kappa_1 M} \sum_{i \in \Omega} \bm a_i\bm a_i^T  -  \bm I_m \right\| \le \frac{9(\sqrt{m}+\eta)}{\sqrt{\kappa_1 M}}
\end{align} 
holds with probability at least $1-2\exp\left(-2\eta^2\right)$. For this fixed $\Omega$, applying Weyl's inequality to \eqref{eq1:cov} yields 
\begin{align}\label{concentrate:omega}
\lambda_{\min}\left( \sum_{i \in \Omega} \bm a_i\bm a_i^T\right) \ge \kappa_1 M \left(1 - \frac{9(\sqrt{m}+\eta)}{\sqrt{\kappa_1 M}}\right).
\end{align}
Obviously, a necessary condition for the lower bound in \eqref{concentrate:omega} to be meaningful is $\eta < \sqrt{\kappa_1 M} / 9$. Therefore, the guaranteed probability that \eqref{concentrate:omega} holds for each fixed $\Omega$ is less than $ 1-2\exp\left(-2\kappa_1 M / 81\right)$. 

On the other hand, the number of $\Omega$ with cardinality equal to $\kappa_1 M$ is given by
\begin{align*}
    C_{M}^{\kappa_1 M} 
    = \frac{M!}{(\kappa_1 M)! ((1 - \kappa_1)M)!}
    \approx & \frac{\sqrt{2\pi M} \left(\frac{M}{e} \right)^M}{\sqrt{2\pi \kappa_1 M} \left(\frac{\kappa_1 M}{e} \right)^{\kappa_1 M} \sqrt{2\pi (1 - \kappa_1)M} \left(\frac{(1 - \kappa_1)M}{e} \right)^{(1 - \kappa_1)M}} \\
    = & \frac{ \exp\left(M \left(\kappa_1 \log\frac{1}{\kappa_1} + (1 - \kappa_1) \log\frac{1}{1 - \kappa_1}\right)  \right)}{\sqrt{2\pi\kappa_1(1 - \kappa_1) M}}, 
\end{align*}
where the approximation holds because of the Stirling approximation $n! \approx \sqrt{2\pi n} \left(n/e\right)^n $. By taking the union bound, the probability becomes 
$$
1-\frac{\sqrt{2}\exp\left(\left(\kappa_1 \log\frac{1}{\kappa_1} + (1 - \kappa_1) \log\frac{1}{1 - \kappa_1} - \frac{2}{81} \kappa_1 \right) M \right)}{\sqrt{\pi\kappa_1(1 - \kappa_1) M}}. 
$$
To ensure the probability is nonnegative, we need at least $\log\left({1}/{\kappa_1}\right) <  2/{81}$, i.e., $\kappa_1 > \exp\left({-{2}/{81}}\right) \approx 0.98$, which is close to 1.

\paragraph{Tightness of the above analysis.} One may question the tightness of the bound on $\kappa_1$ derived above, as the bound in \Cref{lem:cov esti} itself may not be sharp. Because \Cref{lem:cov esti} only provides the lower bound of the probability that \eqref{eq1:cov} holds, the condition $\kappa_1 \ge 0.98$ is a sufficient condition to apply the union bound.  To derive a necessary condition, we instead use the central limit theory to get an upper bound on the probability that \eqref{eq1:cov} holds, and then compute the total probability in the union bound when $m=1$. This analysis gives the necessary condition $\kappa_1 \ge 0.78$. Therefore, the requirement that $\kappa_1 \ge 0.98$ is nearly tight. Below, we present the detailed analysis to derive the necessary condition. 

We first consider a special case where $m = 1$. The central limit theory yields that
\begin{align*}
\frac{1}{\sqrt{2\kappa_1 M}} \sum_{i \in \Omega} (a_i^2 -1)  \overset{d}{\to} \mathcal{N}(0, 1). 
\end{align*}
Thus, for any fixed $\eta \in (0, 1/2]$, there exists some $M_0 >0$, such that for all $M \geq M_0$,
$$
\left| \mathbb{P} \left( \frac{1}{\sqrt{2\kappa_1 M}} \sum_{i \in \Omega} (a_i^2 -1) \geq \eta \right) - (1 -\Phi(\eta)) \right| \leq \frac{1 -\Phi(\eta)}{2}.
$$
That is, for all $M \geq M_0$,
$$
\mathbb{P} \left( \frac{1}{\sqrt{2\kappa_1 M}} \sum_{i \in \Omega} (a_i^2 -1) \geq \eta \right) \geq \frac{1 -\Phi(\eta)}{2}.
$$
According to Mills' ratio of the standard Gaussian distribution \citep{gordon1941values}, we have
$$
1 -\Phi(\eta) \geq \frac{\eta}{\sqrt{2\pi}(1 + \eta^2)} \exp\left(\frac{-\eta^2}{2} \right).
$$
Thus, we have
\begin{align*}
    \ \mathbb{P} \left( \left|\frac{1}{\sqrt{2\kappa_1 M}} \sum_{i \in \Omega} (a_i^2 -1) \right| 
\leq \eta \right)
\leq \ 1 - \frac{\eta}{\sqrt{2\pi}(1 + \eta^2)} \exp\left(\frac{-\eta^2}{2} \right).
\end{align*}
Therefore, the lower bound of the probability provided in \Cref{lem:cov esti} is tight up to a ratio $1/\eta$.

In order to get a valid lower bound, we need to choose $\eta \leq \sqrt{ {\kappa_1 M}/{2}}$, then we have 
\begin{align*}
\frac{\eta}{\sqrt{2\pi}(1 + \eta^2)} \exp\left(\frac{-\eta^2}{2} \right) \ge  \frac{1}{2\sqrt{\pi \kappa_1 M}} \exp\left(\frac{-\kappa_1 M}{4} \right).    
\end{align*}
By taking the union bound, the probability is at most $1-\frac{1}{2\sqrt{2}\pi M}\exp\left(\left(\kappa_1 \log\frac{1}{\kappa_1} + (1 - \kappa_1) \log\frac{1}{1 - \kappa_1} - \frac{1}{4} \kappa_1 \right) M \right). $
To ensure the probability is nonnegative when $M$ is large, we need at least $\log\left({1}/{\kappa_1}\right) <  1/{4}$, i.e., $\kappa_1 > \exp\left({-{1}/{4}}\right) \approx 0.78$. 

\smallskip 
\begin{remark}
    We can also consider the complementary set $\overline{\Omega} = [M] \setminus \Omega$ and compute an upper bound on 
    $\lambda_{\max}\left( \sum_{i \in \overline{\Omega}} \bm a_i\bm a_i^T \right)$ to derive a lower bound for 
    $\lambda_{\min}\left( \sum_{i \in \Omega} \bm a_i\bm a_i^T \right)$. 
    However, this approach leads to the same conclusion.
\end{remark}

\subsection{Proof of \Cref{thm:lowerbound:general}} \label{sec:proof-lowerbound-general}


\begin{proof}

We prove only part~(i), since the proof of part~(ii) is analogous. Using \citep[Corollary 4.4.7]{vershynin2018high} and the fact that the random vectors are isotropic sub-Gaussian with uniformly bounded one-dimensional \(\psi_2\)-norm, we obtain that it holds with probability at least $1 - 2\exp(-t^2)$ that  
$$
\|\bm A\| \leq CK(\sqrt{m} + \sqrt{M} + t),
$$
where $C$ is an absolute constant and $K=\|\bm A_{ij}\|_{\psi_2}$. By choosing $t = \sqrt{M}$, with probability at least $1 - 2\exp(-M)$, we have
$$
\|\bm A\| \leq 3CK\sqrt{M}.
$$
Similar to \eqref{eq5:lem cov A1}, we have
\begin{align}\label{eq5:lem cov A1 5}
\sigma_{\min}\left( \bm A_{\Omega}  \right) \ge \min_{\bm x \in \mathcal{N}_{\varepsilon}} \|\bm A_{\Omega}^T\bm x\| - \varepsilon\|\bm A\| \ge \min_{\bm x \in \mathcal{N}_{\varepsilon}} \|\bm A_{\Omega}^T\bm x\| - 3CK\sqrt{M}\varepsilon. 
\end{align} 
We claim and prove below that it holds with probability defined below in \eqref{eq:prop spec 5} that 
\begin{align}\label{eq3:lem cov A1 5}  
& \min_{\bm x \in \mathcal{N}_{\varepsilon}} \|\bm A_{\Omega}^T\bm x\|^2 \geq (1 -\eta) \mu M,
\end{align}
where $\mu$ is defined in \eqref{eq:gamma sub 1}. Substituting this into \eqref{eq5:lem cov A1 5} and choosing $\varepsilon = \sqrt{\mu} \eta/ (8c_3)$ where $c_3 = CK$ yields 
\begin{align*}
\sigma_{\min}\left( \bm A_{\Omega} \right) \ge \sqrt{(1 -\eta) \mu M} - 3\eta \sqrt{\mu M} / 8 \ge (1 -\eta) \sqrt{\mu M},
\end{align*}
where the last inequality holds because $\eta \leq 1/2$ implies $$
\sqrt{1 -\eta} - (1 -\eta) = \frac{(1 -\eta) - (1 -\eta)^2}{\sqrt{1 -\eta} + (1 -\eta)} = \eta \cdot \frac{1}{1 + 1 /\sqrt{1 -\eta}} \geq \frac{1}{1 +\sqrt{2}} \eta \geq  \frac{3}{8} \eta.
$$
Using the union bound, we obtain the probability of the event. 

The rest of the proof is devoted to proving the claim \eqref{eq3:lem cov A1 5}. We fix $\bm x \in \mathcal{N}_{\epsilon}$ and let  
\begin{align}\label{eq2:lem cov A1}
\mathcal{A} := \left\{ i  \in [M]: \left(\bm a_i^T\bm x\right)^2 \le \gamma^2  \right\},\ \text{where}\ \gamma := \sup\{z\ge0:\Psi_1(z)\le(1-\eta)\kappa_3\}.
\end{align}
According to the definition of $\Psi_1$ in \eqref{eq:Psi1},  we know that
\begin{align}\label{eq4:lem cov A1}
\P\left( \left(\bm a_i^T\bm x\right)^2 \le \gamma^2 \right) = \P\left( -\gamma \le \bm a_i^T\bm x \le \gamma \right) = P_x(\gamma)  - P_x(-\gamma) \le (1-\eta)\kappa_3.
\end{align}
It follows from \eqref{eq2:lem cov A1} that  $|\mathcal{A}| = \sum_{i =1}^M \mathbb{I}\left\{ \left(\bm a_i^T\bm x\right)^2 \le \gamma^2  \right\}$, 
where $\mathbb{I}\left\{ \left(\bm a_i^T\bm x\right)^2 \le \gamma^2  \right\} = 1$ if $\left(\bm a_i^T\bm x\right)^2 \le \gamma^2$ and $0$ otherwise. This, together with \eqref{eq4:lem cov A1}, implies 
\begin{align*}
\E\left[ |\mathcal{A}| \right] \leq (1-\eta)\kappa_3 M ,\quad \mathrm{Var}\left( |\mathcal{A}| \right) \le (1-\eta)\kappa_3 M.
\end{align*}
Applying Bernstein's inequality (see, e.g., \citep[Theorem 2.8.4]{vershynin2018high}) to $|\mathcal{A}|$, which is the sum of i.i.d. Bernoulli random variables, yields that 
\begin{align*}
\P\left( \left| |\mathcal{A}| - \E\left[ |\mathcal{A}| \right] \right| \ge \eta \kappa_3 M \right) & \le 2\exp\left(-\frac{(\eta\kappa_3 M)^2}{2\left( (1 - \eta)\kappa_3 M + \kappa_3 M/{3} \right)}\right)
\le 2\exp\left( -\frac{\eta^2\kappa_3 M}{2} \right). 
\end{align*}
This implies that it holds with probability at least $1-2\exp\left( -{\eta^2\kappa_3 M}/{2} \right)$ that $|\mathcal{A}| \le \kappa_3 M$. Since $|\Omega| \ge \kappa_3 M$, we obtain 
\begin{align}\label{eq:boundsmallerset}
\|\bm A_{\Omega}^T\bm x\|^2 = \sum_{i \in \Omega} \left(\bm a_i^T\bm x\right)^2 \ge \sum_{i \in \mathcal{A}} \left(\bm a_i^T\bm x\right)^2 = \sum_{ i=1 }^M  \left(\bm a_i^T\bm x\right)^2 \mathbb{I}\left\{  \left(\bm a_i^T\bm x\right)^2 \le \gamma^2 \right\}, 
\end{align}
where the last equality follows from \eqref{eq2:lem cov A1}. Since \(\bm a_i^T\bm x\) has distribution \(P_x\), we compute
\begin{align}\label{eq6:lem cov A1 5}
\mu_x & := \E\left[ \left(\bm a_i^T\bm x\right)^2 \mathbb{I}\left\{  \left(\bm a_i^T\bm x\right)^2 \le \gamma^2 \right\} \right] 
\end{align}
where  \(\mu_x\) is bounded below by the \(\mu\) in \eqref{eq:gamma sub 1}. Applying Hoeffding's inequality for bounded random variables yields
\begin{align*}
\P\left( \left|\sum_{i=1}^M \left(\bm a_i^T\bm x\right)^2 \mathbb{I}\left\{  \left(\bm a_i^T\bm x\right)^2 \le \gamma^2 \right\} - M \mu_x \right| \ge \eta M \mu \right) \le 2\exp\left( - \frac{\eta^2\mu^2 M}{2\gamma^4} \right).  
\end{align*}
This implies that it holds with probability at least $1-2\exp\left( - {\eta^2\mu^2 M}/{(2\gamma^4)} \right)$ that 
\begin{align*}
 \sum_{i = 1}^M \left(\bm a_i^T\bm x\right)^2 \mathbb{I}\left\{  \left(\bm a_i^T\bm x\right)^2 \le \gamma^2 \right\}  \ge  (1-\eta) M \mu_x, 
\end{align*} 
Finally, using the union bound and taking $\mu = \min_{x \in \mathbb S^{m-1}} \mu_x$, the result holds with probability at least 
\begin{align}\label{eq:prop spec 5}
1 - 2 \left( \frac{16c_3}{\eta\sqrt{\mu } } + 1 \right)^m \left( \exp\left( -\frac{\eta^2\kappa_3 M}{2} \right)  +  \exp\left( -\frac{\eta^2\mu^2 M}{2\gamma^4} \right) \right) .
\end{align}

\end{proof}

\section{Proofs for \Cref{sec:ext}}\label{app:ext}

\subsection{Gaussian Sequence with Geometrically Decaying Covariance}\label{app:ext 1}

\begin{proposition}\label{prop:gaussian-cov-decay-mixing}
Let $\{\bm a_i\}_{i\in\mathbb Z}\subseteq\mathbb R^m$ be a centered stationary
jointly Gaussian sequence such that $\bm a_i\sim \mathcal N(0,\bm I_m),\ i\in\mathbb Z.$
Let $\Gamma(k):=\operatorname{Cov}(\bm a_0,\bm a_k).$ Suppose that there exist constants $C_\Gamma>0$ and
$\rho_\Gamma\in(0,1)$ such that
\[
\|\Gamma(k)\|\le C_\Gamma \rho_\Gamma^k,
\qquad k\ge 1.
\]
Then $\{\bm a_i\}_{i\in\mathbb Z}$ is geometrically strong mixing.
\end{proposition}

\begin{proof}
Thanks to the stationarity, it is enough to show that
\[
\alpha(n)
:=
\sup_{A\in\mathcal F_{-\infty}^{0},\,
B\in\mathcal F_{n}^{\infty}}
\left|
\mathbb P(A\cap B)-\mathbb P(A)\mathbb P(B)
\right|
\le
C_\alpha \rho_\alpha^n,
\qquad n\ge 1,
\]
Since the sequence is Gaussian, the dependence between the past
$\mathcal F_{-\infty}^{0}$ and the future $\mathcal F_n^\infty$ is fully
determined by the corresponding covariance structure. We use the following
standard Gaussian comparison bound: for stationary Gaussian sequences, there
exists a constant $C_m>0$, depending only on the fixed dimension $m$, such that
\[
\alpha(n)
\le
C_m
\sum_{r=0}^{\infty}
\sum_{s=0}^{\infty}
\left\|
\operatorname{Cov}(\bm a_{-r},\bm a_{n+s})
\right\|.
\]
By stationarity,
\[
\operatorname{Cov}(\bm a_{-r},\bm a_{n+s})
=
\operatorname{Cov}(\bm a_0,\bm a_{n+r+s})
=
\Gamma(n+r+s).
\]
Therefore,
\[
\alpha(n)
\le
C_m
\sum_{r=0}^{\infty}
\sum_{s=0}^{\infty}
\|\Gamma(n+r+s)\|.
\]
Using the geometric covariance decay assumption,
\[
\|\Gamma(n+r+s)\|
\le
C_\Gamma \rho_\Gamma^{n+r+s}.
\]
Hence
\[
\begin{aligned}
\alpha(n)
\le
C_m C_\Gamma
\sum_{r=0}^{\infty}
\sum_{s=0}^{\infty}
\rho_\Gamma^{n+r+s}  =
C_m C_\Gamma \rho_\Gamma^n
\left(\sum_{r=0}^{\infty}\rho_\Gamma^r\right)
\left(\sum_{s=0}^{\infty}\rho_\Gamma^s\right) =
\frac{C_m C_\Gamma}{(1-\rho_\Gamma)^2}
\rho_\Gamma^n.
\end{aligned}
\]
Thus, by setting
\[
C_\alpha:=\frac{C_m C_\Gamma}{(1-\rho_\Gamma)^2},
\qquad
\rho_\alpha:=\rho_\Gamma,
\]
we obtain
\[
\alpha(n)\le C_\alpha \rho_\alpha^n,
\qquad n\ge 1.
\]
This proves geometric strong mixing.
\end{proof}

\subsection{Proof of \Cref{thm:mixing-lower-bound}}\label{app:ext 2}

We first establish the following operator norm bound for geometrically strong mixing sequence. 
 To prove this theorem, the key step is to derive concentration bounds for the following two empirical quantities associated with each fixed direction \(\bm x\in\mathbb S^{m-1}\):
\[
Y_i^{(\bm x,\gamma)}:=\bm 1\{|\bm a_i^T \bm x|\le \gamma\},
\qquad
Z_i^{(\bm x,\gamma)}:=(\bm a_i^T \bm x)^2\bm 1\{|\bm a_i^T \bm x|\le \gamma\}.
\]

Then the fixed-direction concentration step can be replaced by Bernstein-type inequalities for the bounded and geometrically strong mixing sequences \cite[Theorem 1]{merlevede2009bernstein}: 
\begin{align}
\mathbb P\!\left(
\left|
\sum_{i=1}^M \big(Y_i^{(\bm x,\gamma)}-\mathbb E Y_i^{(\bm x,\gamma)}\big)
\right|\ge t
\right)
&\le
\exp\!\left(
-\frac{ ct^2}{M + t (\log M) (\log \log M)}
\right), \label{eq:mixing-Y}\\
\mathbb P\!\left(
\left|
\sum_{i=1}^M \big(Z_i^{(\bm x,\gamma)}-\mathbb E Z_i^{(\bm x,\gamma)}\big)
\right|\ge t
\right)
&\le
\exp\!\left(
-\frac{ ct^2}{M\gamma^4 + \gamma^2 t (\log M) (\log \log M)}
\right).  \label{eq:mixing-Z}
\end{align}

\begin{lemma}
\label{lem:mixing-opnorm-gaussian}
Let $\{\bm a_i\}_{i=1}^M\subseteq\mathbb R^m$ be a standard Gaussian and geometrically strong mixing sequence, i.e., there exist
constants $C_\alpha>0$ and $\rho\in(0,1)$ such that
\[
\bm a_i\sim \mathcal N(0,I_m),\quad i=1,\ldots,M, \qquad 
\alpha(k)\le C_\alpha \rho^k,\qquad k\ge 1.
\]
Then there exist constants $B>0$ and $c>0$, depending only on
$C_\alpha$ and $\rho$, such that whenever $M\ge c m$, it holds with probability
at least $1-\exp(-cM)$
that
\[
\left\|
\sum_{i=1}^M \bm a_i \bm a_i^T
\right\|
\le BM.
\]
\end{lemma}

\begin{proof}
We first show that geometric strong mixing implies geometric decay of the
cross-covariance matrices. Fix \(i,j\in[M]\) and \(\bm u,\bm v\in\mathbb S^{m-1}\).
Define
\[
X:=\bm u^T \bm a_i,\qquad Y:=\bm v^T \bm a_j.
\]
Then \((X,Y)\) is a centered bivariate Gaussian vector with unit variances and
correlation
\[
r_{ij}(\bm u,\bm v)
:=
\mathbb E[(\bm u^T \bm a_i)(\bm v^T \bm a_j)]
=
\bm u^T \operatorname{Cov}(\bm a_i,\bm a_j)\bm v.
\]
For bivariate Gaussian random variables, the quadrant-event identity gives
\[
\mathbb P(X\ge 0,Y\ge 0)-\frac14
=
\frac{1}{2\pi}\arcsin r_{ij}(\bm u,\bm v).
\]
Hence
\[
\alpha(\sigma(X),\sigma(Y))
\ge
\frac{1}{2\pi}|\arcsin r_{ij}(\bm u,\bm v)|.
\]
Since \(\sigma(X)\subseteq \sigma(\bm a_i)\) and \(\sigma(Y)\subseteq\sigma(\bm a_j)\),
the strong-mixing assumption implies
\[
\alpha(\sigma(X),\sigma(Y))
\le
\alpha(|i-j|)
\le
C_\alpha \rho^{|i-j|}.
\]
Therefore,
\[
|\arcsin r_{ij}(\bm u,\bm v)|
\le
2\pi C_\alpha \rho^{|i-j|}.
\]
Consequently, there exist constants \(C_1>0\) and \(\rho_1\in(0,1)\), depending
only on \(C_\alpha\) and \(\rho\), such that
\[
|r_{ij}(\bm u,\bm v)|
\le
C_1\rho_1^{|i-j|}.
\]
Taking the supremum over \(\bm u,\bm v\in\mathbb S^{m-1}\) yields
\[
\|\operatorname{Cov}(\bm a_i,\bm a_j)\|
\le
C_1\rho_1^{|i-j|}.
\]

Next, fix \(\bm x\in\mathbb S^{m-1}\) and define
\[
\bm b_{\bm x}:=(\bm a_1^T \bm x,\ldots,\bm a_M^T \bm x)^T\in\mathbb R^M.
\]
Then \(\bm b_{\bm x}\) is a centered Gaussian vector with covariance matrix
\(\bm\Sigma_{\bm x}\in\mathbb R^{M\times M}\) satisfying
\[
(\bm\Sigma_{\bm x})_{ij}
=
\bm x^T \operatorname{Cov}(\bm a_i,\bm a_j)\bm x.
\]
Since \(\bm a_i\sim\mathcal N(\bm 0,\bm I_m)\), we have \(\operatorname{tr}(\bm\Sigma_{\bm x})=M\).
Moreover, by the covariance decay just proved,
\[
\|\bm\Sigma_{\bm x}\|
\le
\max_{i\in[M]}\sum_{j=1}^M |(\bm\Sigma_{\bm x})_{ij}|
\le
1+2C_1\sum_{h=1}^{\infty}\rho_1^h
=:L.
\]
Therefore,
\[
\operatorname{tr}(\bm\Sigma_{\bm x}^2)
\le
\|\bm\Sigma_{\bm x}\|\operatorname{tr}(\bm\Sigma_{\bm x})
\le
LM.
\]

By the Gaussian quadratic-form concentration inequality, for every \(t>0\),
\[
\mathbb P\left(
\|\bm b_{\bm x}\|^2
\ge
M+2\sqrt{LMt}+2Lt
\right)
\le
e^{-t}.
\]
Equivalently,
\[
\mathbb P\left(
\sum_{i=1}^M (\bm a_i^T \bm x)^2
\ge
M+2\sqrt{LMt}+2Lt
\right)
\le
e^{-t}.
\]

Let \(\mathcal N\) be a \(1/4\)-net of \(\mathbb S^{m-1}\). Then \(|\mathcal N|\le 9^m\).
Taking \(t=c_0M\) for a sufficiently large constant \(c_0>0\) and applying the
union bound, we obtain
\[
\max_{\bm x\in\mathcal N}
\sum_{i=1}^M (\bm a_i^T \bm x)^2
\le
C_LM
\]
with probability at least \(1-9^m e^{-c_0M}\).
If \(M\ge c m\) for a sufficiently large constant \(c>0\), then $9^m e^{-c_0M}\le \exp(-cM)$ after adjusting constants.

Finally, by the standard net argument for symmetric matrices,
\[
\left\|
\sum_{i=1}^M \bm a_i \bm a_i^T
\right\|
\le
2
\max_{\bm x\in\mathcal N}
\bm x^T
\left(
\sum_{i=1}^M \bm a_i \bm a_i^T
\right)
\bm x.
\]
Therefore,
\[
\left\|
\sum_{i=1}^M \bm a_i \bm a_i^T
\right\|
\le
2C_LM.
\]
Setting \(B:=2C_L\) completes the proof.
\end{proof}

\begin{proof}[Proof of \Cref{thm:mixing-lower-bound}]
Let \(\bm A\in\mathbb R^{m\times M}\) be the matrix whose columns are
\(\bm a_1,\ldots,\bm a_M\), and let \(\bm A_\Omega\) be the submatrix whose columns are
\(\{\bm a_i:i\in\Omega\}\). 

Let \(\mathcal N_\epsilon\) be an \(\epsilon\)-net of \(\mathbb S^{m-1}\) with
\[
|\mathcal N_\epsilon|
\le
\left(\frac{2}{\epsilon}+1\right)^m.
\]
For any \(\bm x\in\mathbb S^{m-1}\), choose \(\bm x_0\in\mathcal N_\epsilon\) such that
\(\|\bm x-\bm x_0\|\le \epsilon\). On the event \(\mathcal E_{\rm op}\), we have
\[
\|\bm A_\Omega^T \bm x\|
\ge
\|\bm A_\Omega^T \bm x_0\|-\epsilon\|\bm A_\Omega\|
\ge
\|\bm A_\Omega^T \bm x_0\|-\epsilon\|\bm A\|
\ge
\|\bm A_\Omega^T \bm x_0\|-\epsilon\sqrt{BM},
\]
where the last inequality follows from \Cref{lem:mixing-opnorm-gaussian}, which holds with probability \(1-\exp(-cM)\). Therefore,
\[
\sigma_{\min}(\bm A_\Omega)
\ge
\min_{\bm x\in\mathcal N_\epsilon}\|\bm A_\Omega^T \bm x\|
-\epsilon\sqrt{BM}.
\]

We next prove a uniform lower bound over the net. Fix \(\bm x\in\mathcal N_\epsilon\)
and define
\[
\mathcal A_{\bm x}:=
\left\{
i\in[M]: |\bm a_i^T \bm x|\le \gamma
\right\}.
\]
Since \(\bm a_i^T \bm x\sim \mathcal N(0,1)\), we have
\[
\mathbb E|\mathcal A_{\bm x}|
=
M\mathbb P(|g|\le \gamma)
=
(1-\eta)\kappa M.
\]
Applying \eqref{eq:mixing-Y} with \(t=\eta\kappa M\) yields
\[
\mathbb P\left(|\mathcal A_{\bm x}|\ge \kappa M\right)
\le
2\exp\left(-c\eta^2\kappa^2M_{\rm eff}\right),
\]
where the constant \(c>0\) may change from line to line. Hence, on the
complementary event, \(|\mathcal A_{\bm x}|\le \kappa M\).

Since \(|\Omega|\ge \kappa M\), and \(\mathcal A_{\bm x}\) consists of the indices with
the smallest projected magnitudes below the threshold \(\gamma\), we have
\[
\sum_{i\in\Omega}(\bm a_i^T \bm x)^2
\ge
\sum_{i\in\mathcal A_{\bm x}}(\bm a_i^T \bm x)^2
=
\sum_{i=1}^{M}
(\bm a_i^T \bm x)^2\bm 1\{|\bm a_i^T \bm x|\le \gamma\}.
\]
Moreover,
\[
\mathbb E\left[
(\bm a_i^T \bm x)^2\bm 1\{|\bm a_i^T \bm x|\le \gamma\}
\right]
=
\mu.
\]
Applying \eqref{eq:mixing-Z} with \(t=\eta\mu M\) gives
\[
\mathbb P\left(
\sum_{i=1}^{M}
(\bm a_i^T \bm x)^2\bm 1\{|\bm a_i^T \bm x|\le \gamma\}
\le
(1-\eta)\mu M
\right)
\le
2\exp\left(
-\frac{c\eta^2\mu^2}{\gamma^4}M_{\rm eff}
\right).
\]
Thus, for this fixed \(\bm x\in\mathcal N_\epsilon\), with probability at least
\[
1
-
2\exp\left(-c\eta^2\kappa^2M_{\rm eff}\right)
-
2\exp\left(
-\frac{c\eta^2\mu^2}{\gamma^4}M_{\rm eff}
\right),
\]
we have
\[
\|\bm A_\Omega^T \bm x\|^2
=
\sum_{i\in\Omega}(\bm a_i^T \bm x)^2
\ge
(1-\eta)\mu M.
\]
Taking a union bound over all \(\bm x\in\mathcal N_\epsilon\) gives
\[
\min_{\bm x\in\mathcal N_\epsilon}
\|\bm A_\Omega^T \bm x\|^2
\ge
(1-\eta)\mu M
\]
with probability at least
\[
1
-
2|\mathcal N_\epsilon|
\left[
\exp\left(-c\eta^2\kappa^2M_{\rm eff}\right)
+
\exp\left(
-\frac{c\eta^2\mu^2}{\gamma^4}M_{\rm eff}
\right)
\right].
\]

It remains to choose the net radius. Let $\epsilon := {\eta\sqrt{\mu}}/({3\sqrt B})$. Then, we have 
\[
|\mathcal N_\epsilon|
\le
\left(
\frac{6\sqrt B}{\eta\sqrt\mu}+1
\right)^m.
\]
On the intersection of the above event and \(\mathcal E_{\rm op}\), we obtain
\[
\sigma_{\min}(\bm A_\Omega)
\ge
\sqrt{(1-\eta)\mu M}
-
\frac{\eta\sqrt{\mu M}}{3}.
\]
Since \(\eta\in(0,1/2)\), we have
\[
\sqrt{1-\eta}-\frac{\eta}{3}\ge 1-\eta.
\]
Therefore,
\[
\sigma_{\min}(\bm A_\Omega)
\ge
(1-\eta)\sqrt{\mu M}.
\]
Squaring both sides yields
\[
\lambda_{\min}
\left(
\sum_{i\in\Omega}\bm a_i \bm a_i^T
\right)
=
\sigma_{\min}^2(\bm A_\Omega)
\ge
(1-\eta)^2\mu M.
\]
The desired probability follows by combining the union bound with the probability with which \Cref{lem:mixing-opnorm-gaussian} holds.
\end{proof}

\section{Proofs for \Cref{sec:KSS}}\label{sec:proof:errorrate}

\subsection{Proof of \Cref{prop:opti}}\label{proof:prop opti}

\begin{proof}
For fixed \(C\), minimizing \(F(\bm U,C)\) over each \(\bm U_k\in\mathcal O^{n\times d}\) is equivalent to maximizing
\[
\sum_{i\in C_k}\|\bm U_k^T\bm z_i\|^2
=
\left\langle \bm U_k\bm U_k^T,\sum_{i\in C_k}\bm z_i\bm z_i^T\right\rangle.
\]
Hence \(\bm U_k\) must span a top-\(d\) eigenspace of \(\sum_{i\in C_k}\bm z_i\bm z_i^T\). For fixed \(\bm U\), each data point must be assigned to a cluster maximizing \(\|\bm U_k^T\bm z_i\|\); otherwise moving that point to a cluster with a larger projected norm strictly decreases the objective, up to arbitrary tie-breaking.  
\end{proof}

\subsection{Proof of \Cref{lem:optimal}}\label{app:prop opt}
\begin{proof}
Suppose that \eqref{eq:norm ae} holds, which happens with probability at least $1-4N^{-1}$ according to \Cref{lem:norm ae}. 
Since $\bm U_k^T\bm U_k = \bm I$ for all $k \in [K]$, we have $F(\bm U, C)$ $= \sum_{k=1}^K \sum_{i \in C_k} \left(\|\bm z_i\|^2 -  \|\bm U_k^T\bm z_i\|^2\right).$ For any optimal solution $(\hat{\bm U}, \hat{C})$ of Problem \eqref{eq:loss}, it holds that $F(\hat{\bm U}, \hat{C}) \le F(\bm U^\star, C^\star)$. According to $\bm z_i = \bm U_k^\star \bm a_i + \bm e_i$ for all $i \in C_k^\star$, we compute 
\begin{align*} 
	F(\hat{\bm U}, \hat{C}) - F(\bm U^\star, C^\star) & = \sum_{k=1}^K \sum_{i \in C^\star_k}  \|\bm U_k^{\star T}\bm z_i\|^2 - \sum_{k=1}^K \sum_{i \in \hat{C}_k}  \|\hat{ \bm U}_k^T\bm z_i\|^2 \\
	& =  \sum_{k=1}^K \sum_{i \in C^\star_k}  \|\bm a_i + \bm U_k^{\star T} \bm e_i\|^2 - \sum_{k=1}^K\sum_{l=1}^K \sum_{i \in \hat{C}_k \cap C_l^\star} \|\hat{ \bm U}_k^T\left( \bm U_l^\star \bm a_i + \bm e_i \right) \|^2 \\
	& = \sum_{i = 1}^N \|\bm a_i\|^2 - \sum_{k=1}^K\sum_{l=1}^K \sum_{i \in \hat{C}_k  \cap C_l^\star}   \|\hat{\bm U}_k^T\bm U_l^\star\bm a_i\|^2 - E,
\end{align*}
where \(E\) in the last equality is defined as
\begin{align}\label{eq:E}
E := & \sum_{k=1}^K\sum_{l=1}^K  \sum_{i \in \hat{C}_k  \cap C_l^\star} 2\langle \bm a_i, \bm U_l^{\star T} \hat{\bm U}_k\hat{\bm U}_k^T \bm e_i \rangle - 2 \sum_{k=1}^K \sum_{i \in C_k^\star} \langle \bm a_i, \bm U_k^{\star T}\bm e_i \rangle + \notag \\
&\ \sum_{k=1}^K \sum_{i \in \hat{C}_k } \|\hat{\bm U}_k^T\bm e_i\|^2 - \sum_{k=1}^K \sum_{i \in C_k^\star} \|\bm U_k^{\star T}\bm e_i\|^2. 
\end{align}
This, together with $F(\hat{\bm U}, \hat{C}) \le F(\bm U^\star, C^\star)$, implies 
\begin{align}\label{eq1:prop rate}
&\ \sum_{k=1}^K \sum_{l=1}^K \langle \bm I -  \bm U_l^{\star T}\hat{\bm U}_k\hat{\bm U}_k^T  \bm U_l^{\star}, \sum_{i \in \hat{C}_k\cap C_l^\star} \bm a_i\bm a_i^T \rangle \notag \\
= &\ \sum_{i = 1}^N \|\bm a_i\|^2 -  \sum_{k=1}^K\sum_{l=1}^K \sum_{i \in \hat{C}_k  \cap C_l^\star} \|\hat{\bm U}_k^T\bm U_l^\star\bm a_i\|^2 \le E. 
\end{align}
According to \eqref{eq:E}, we bound 
\begin{align}\label{eq2:prop rate}
E \le 4 \sum_{k=1}^K  \sum_{i \in C_k^\star} \|\bm a_i\| \|\bm e_i\| + 2\sum_{k=1}^K \sum_{i\in C_k^\star} \|\bm e_i\|^2  \le 9N\left( \delta_{\max} \sqrt{nd} + \delta_{\max}^2 n \right),
\end{align}
where the second inequality follows from \eqref{eq:norm ae} and $d \gtrsim (\sqrt{\log N} + 1)^2$ due to \eqref{eq:ndk}. 
\end{proof}

\subsection{Proof of \Cref{prop:error rate}}\label{app:prop error rate}

\begin{proof}
We first show why \eqref{ineq:triangle} is true. Notice that
$$
d - \|\bm U_{l_1}^{\star T}\bm U_{l_2}^\star\|_F^2 = \sum_{i = 1}^d 1 - \sigma_i(\bm U_{l_1}^{\star T}\bm U_{l_2}^\star)^2 \quad \text{and} \quad \min_{\bm Q \in \mathcal{O}(d)} \|\bm U_{l_1}^{\star} - \bm U_{l_2}^\star \bm Q\|^2 = 2\sum_{i = 1}^d 1 - \sigma_i(\bm U_{l_1}^{\star T}\bm U_{l_2}^\star) 
$$
We know that
\begin{align*}
    d - \|\bm U_{l_1}^{\star T}\bm U_{l_2}^\star\|_F^2 
    \leq \min_{\bm Q \in \mathcal{O}(d)} \|\bm U_{l_1}^{\star} - \bm U_{l_2}^\star \bm Q\|^2 
    \leq & 2 \min_{\bm Q \in \mathcal{O}(d)} \min_{\bm Q' \in \mathcal{O}(d)}\left( \|\bm U_{l_1}^{\star} - \hat{\bm U}_{k} \bm Q'\|^2 + \| \hat{\bm U}_{k} \bm Q' - \bm U_{l_2}^\star \bm Q\|^2 \right) \\
    = & 2 \min_{\bm Q' \in \mathcal{O}(d)} \|\bm U_{l_1}^{\star} - \hat{\bm U}_{k} \bm Q'\|^2 + 2\min_{\bm Q\in \mathcal{O}(d)}\| \hat{\bm U}_{k} - \bm U_{l_2}^\star \bm Q\|^2 \\
    \leq & 4\sum_{i = 1}^d 1 - \sigma_i(\bm U_{l_1}^{\star T}\hat{\bm U}_{k}) + 4\sum_{i = 1}^d 1 - \sigma_i(\hat{\bm U}_{k}^T \bm U_{l_2}^\star) \\
    \leq & 4 \left( d - \|\hat{\bm U}_k^T\bm U_{l_1}^\star\|_F^2  +  d - \|\hat{\bm U}_k^T\bm U_{l_2}^\star\|_F^2 \right).
\end{align*}
Since we assume $\delta_{\max} \leq \frac{1}{40} \sqrt{\frac{d}{n}}$, according to \eqref{eq:delta 1}, we have
$$
E = 9N\left( \delta_{\max} \sqrt{nd} + \delta_{\max}^2 n \right) < 10 N \delta_{\max}  \sqrt{nd} \leq \frac{\mu (1-\eta)^2 (1-\mu_{\max}^2) N_{\min}^\star d}{4 } .
$$
Then, \eqref{ineq:triangle} implies that
$$
d - \|\bm U_{l_1}^{\star T}\bm U_{l_2}^\star\|_F^2 \leq \frac{4E}{(1-\eta)^2\mu N_{\min}^\star} < (1-\mu_{\max}^2) d, 
$$
which contradicts the assumption that $d - \|\bm U_{l_1}^{\star T}\bm U_{l_2}^\star\|_F^2 \geq (1-\mu_{\max}^2)d$. Thus, the scenario as stated in \eqref{ineq:countercase} will not happen. That is, for any $k \in [K]$, there exists at most one $l\in [K]$  such that
\begin{equation}
    |\hat{C}_k \cap C_{l}^\star| \geq \kappa_1 N_{l}^\star.
\end{equation}
Besides, for each $l \in [K]$, since $C_{l}^\star = \cup_{k \in [K]} (\hat{C}_k \cap C_{l}^\star)$ and $\kappa_1 \in (0, 1/{K}]$, we know that there exists at least one $k \in [K]$ such that $|\hat{C}_k \cap C_{l}^\star| \geq \kappa_1 N_{l}^\star$. Thus, according to the pigeonhole principle, there is a permutation $\pi:[K] \to [K]$ such that
$$
|\hat{C}_{\pi(k)} \cap C_k^\star| \ge \kappa_1 N_k^\star,\ \forall k \in [K].
$$
In other words, we have $|\hat{C}_{l} \cap C_k^\star| < \kappa_1 N_k^\star$ for all $l \neq \pi(k)$, which means that
$$
|\hat{C}_{\pi(k)} \cap C_k^\star| \ge \left(1 - (K-1) \kappa_1 \right) N_k^\star,\ \forall k \in [K].
$$
\end{proof}

\subsection{Proof of \Cref{thm:opti 1}}\label{app:proof thm}

\begin{proof}
Based on \Cref{prop:error rate}, we choose \(\kappa_1\asymp(\sqrt{n/d}\,\kappa_NK/\mathrm{SNR})^{1/3}\); then \eqref{eq:delta 1} follows from \eqref{eq:SNR1} and the lower bound \(\mu\gtrsim \kappa_1^3\). It remains to justify \eqref{ineq:kappamu3}. Actually, we only need to show it holds as $\kappa_1$ goes to 0, since there is a constant bound $\mu\geq \mu_0$ for all $\kappa_1 \geq \epsilon$. Recall that 
$$
\mu = \sqrt{\frac{2}{\pi}}\int_0^{\gamma} x^2 \exp\left( -\frac{x^2}{2} \right) dx = \frac{1}{3} \sqrt{\frac{2}{\pi}} \gamma^3 + o(\gamma^3)
$$
and 
$$
\frac{(1-\eta)\kappa_1}{2} = \frac{1}{\sqrt{2\pi}}\int_0^{\gamma} \exp\left( -\frac{x^2}{2} \right) dx = \frac{1}{\sqrt{2\pi}} \gamma + o(\gamma).
$$
Then, we compute
$$
\lim_{\kappa_1 \rightarrow 0} \frac{\mu}{\kappa_1^3} = \lim_{\gamma \rightarrow 0} \frac{ \frac{1}{3} \sqrt{\frac{2}{\pi}} \gamma^3}{\left(\sqrt{\frac{2}{\pi}} \frac{\gamma}{1-\eta}\right)^3} = \frac{(1-\eta)^3\pi}{6}.
$$
\end{proof}

\section{Auxiliary Lemmas}

In this section, we present some concentration inequalities for random vectors, which play an important role in our analysis. We first present a tail bound for the weighted sum of squared Gaussian random variables from its mean, which is proved in \citep[Lemma 1]{laurent2000adaptive}. 
\begin{lemma}\label{lem:chi square}
Let $\bx \sim \mN(\bm 0, \bI_d)$ be a Gaussian random vector and $\lambda_1,\dots,\lambda_d > 0$ be constants. It holds for any $t > 0$ that  
\begin{align*}
\P\left( \left|\sum_{i=1}^d \lambda_i^2 (x_i^2 - 1)\right| \ge 2\sqrt{\sum_{i=1}^d \lambda_i^4 t}  + 2 \lambda_{\max}^2 t   \right) \le 2\exp\left(-t\right).
\end{align*}
\end{lemma}

Then, we present a lemma that provides a concentration inequality on the sum of Lipschitz functions of Gaussian random variables, which can be found in \cite{ledoux2001concentration}. 
\begin{lemma}\label{lem:Lip gaus}
Let $\bm x \sim \mN(\bm 0, \bm I_d)$ be a Gaussian random vector and $f:\R^d \to \R$ be $L$-Lipschitz with respect to the Euclidean norm. Then, we have
\begin{align*}
\P\left(\left|f(\bm x) - \E[f(\bm x)]\right| \ge t \right) \le 2\exp\left( -\frac{t^2}{2L^2} \right),\ \forall t \ge 0.
\end{align*}
\end{lemma}

Based on the above lemma, we present a probabilistic bound on the deviation of the norm of a weighted sum of squared Gaussian random variables from its mean. This is an extension of \citep[Lemma 7]{pmlr-v139-li21f}.  

\begin{lemma}\label{lem:norm gaus}
Let $\bx \sim \mN(\bm 0, \bI_d)$ be a Gaussian random vector and $\lambda_1,\dots,\lambda_d > 0$ be constants. It holds for any $t > 0$ that  
\begin{align}\label{eq:norm gaus}
\P\left(  \left|\sqrt{\sum_{i=1}^d \lambda_i^2x_i^2} - \sqrt{\sum_{i=1}^d \lambda_i^2}\right| \ge t + 2 \lambda_{\max} \right) \le 2\exp\left( -\frac{t^2}{2\lambda_{\max}^2} \right),
\end{align}
where $\lambda_{\max} = \max\{\lambda_i: i \in [d] \}$. 
\end{lemma}
\begin{proof} 
We define  $f(\bx) = \sqrt{\sum_{i=1}^d \lambda_i^2x_i^2}/\sqrt{\sum_{i=1}^d \lambda_i^2}$. By calculation, we obtain
\begin{align*}
 \|\nabla f(\bx)\| = \frac{1}{\sqrt{\sum_{i=1}^d \lambda_i^2}}\sqrt{\frac{\sum_{i=1}^d\lambda_i^4x_i^2}{\sum_{i=1}^d\lambda_i^2x_i^2}} \le \frac{\lambda_{\max}}{\sqrt{\sum_{i=1}^d \lambda_i^2}},
\end{align*}
where $\lambda_{\max} = \max\{\lambda_i: i\in [d]\}$. Applying \Cref{lem:Lip gaus} to $f(\bx)$ yields that 
\begin{align*}
\P\left( \left| f(\bx) - \E[f(\bx)] \right| \ge \frac{t}{\sqrt{\sum_{i=1}^d \lambda_i^2}} \right) \le 2\exp\left( -\frac{t^2}{2\lambda_{\max}^2} \right).  
\end{align*}
This implies 
\begin{align}\label{eq1:lem-norm-gaus}
\P\left( \left| \sqrt{\sum_{i=1}^d \lambda_i^2x_i^2} - \E\left[\sqrt{\sum_{i=1}^d \lambda_i^2x_i^2}\right] \right| \ge t \right) \le 2\exp\left( -\frac{t^2}{2\lambda_{\max}^2} \right). 
\end{align}
We first note that
\begin{align}\label{eq2:lem-norm-gaus}
\E\left[\sqrt{\sum_{i=1}^d \lambda_i^2x_i^2}\right] \le \sqrt{\E\left[\sum_{i=1}^d \lambda_i^2x_i^2\right]} = \sqrt{\sum_{i=1}^d\lambda_i^2}.
\end{align}
By letting $X := \sqrt{\sum_{i=1}^d \lambda_i^2x_i^2} \ge 0$ and $\mu := \E[X]$, we compute
\begin{align*}
\mathrm{Var}\left( X \right) & = \E\left[(X-\mu)^2 \right]  =  \int_0^\infty x^2 d\P\left(  \left|X-\mu  \right|\le x \right)   = -\int_0^\infty x^2 d\P\left(  \left|X-\mu  \right| > x \right)\\
&  = \int_0^\infty 2x \P\left(  \left|X-\mu  \right| > x \right) dx \le  \int_0^\infty 4x \exp\left( -\frac{x^2}{2\lambda_{\max}^2} \right) dx = 4\lambda_{\max}^2,
\end{align*} 
where the fourth equality and the last one follow from integration by parts and the inequality is due to \eqref{eq1:lem-norm-gaus}. Thus, we have
\begin{align*}
\mu^2 & = \E\left[X^2\right] - \mathrm{Var}\left( X \right) = \E\left[ \sum_{i=1}^d \lambda_i^2x_i^2 \right] - \mathrm{Var}\left( X \right) \ge \sum_{i=1}^d\lambda_i^2 - 4\lambda_{\max}^2 .
\end{align*}
Suppose that $\sum_{i=1}^d \lambda_i^2 \ge 4\lambda_{\max}^2$.  This implies
\begin{align}\label{eq3:lem-norm-gaus}
\E[X] \ge \sqrt{ \sum_{i=1}^d\lambda_i^2 - 4\lambda_{\max}^2 } \ge  \sqrt{\sum_{i=1}^d\lambda_i^2} - 2\lambda_{\max}. 
\end{align}
Plugging \eqref{eq2:lem-norm-gaus} and \eqref{eq3:lem-norm-gaus} into \eqref{eq1:lem-norm-gaus} yields that \eqref{eq:norm gaus} holds when $\sum_{i=1}^d \lambda_i^2 \ge 4\lambda_{\max}^2$.

On the other hand, suppose that  $\sum_{i=1}^d \lambda_i^2 < 4\lambda_{\max}^2$. Let $\hat{\lambda}_i = 2\lambda_{\max} \lambda_i / \sqrt{\sum_{i=1}^d \lambda_i^2}$ for each $i \in [d]$. Then, one can verify that $\sum_{i=1}^d \hat{\lambda}_i^2 = 4\lambda_{\max}^2$ and $\max\{\hat{\lambda}_i: i \in [d]\} \le 2\lambda_{\max}^2 / \sqrt{\sum_{i=1}^d \lambda_i^2}$. This, together with the above result, implies that 
\begin{align*}
\P\left(  \left|\sqrt{\sum_{i=1}^d \hat{\lambda}_i^2x_i^2} - \sqrt{\sum_{i=1}^d \hat{\lambda}_i^2}\right| \ge t +  \frac{4\lambda_{\max}^2}{\sqrt{\sum_{i=1}^d \lambda_i^2}} \right) \le 2\exp\left( -\frac{t^2\sum_{i=1}^d \lambda_i^2}{8\lambda_{\max}^4} \right). 
\end{align*}
Therefore, we obtain 
\begin{align*}
&\ \P\left(  \left|\sqrt{\sum_{i=1}^d \lambda_i^2x_i^2} - \sqrt{\sum_{i=1}^d \lambda_i^2}\right| \ge \frac{t\sqrt{\sum_{i=1}^d \lambda_i^2}}{2\lambda_{\max}} + 2\lambda_{\max}  \right) \le 2\exp\left( -\frac{t^2\sum_{i=1}^d \lambda_i^2}{8\lambda_{\max}^4} \right).
\end{align*}
Replacing $t$ by $2\lambda_{\max} t/\sqrt{\sum_{i=1}^d \lambda_i^2}$, we directly obtain \eqref{eq:norm gaus}. 
\end{proof} 

To begin, we present a lemma that estimates the norms of Gaussian random vectors, the inner product of Gaussian random vectors, and the product of a matrix by a Gaussian random vector. 

\begin{lemma}\label{lem:norm ae}
Consider the setting of the MoG model in Definition \ref{def:MoG}. Let \(\alpha:=2\sqrt{\log N}+2\). The following statements hold: \\
(i)  For all $i \in C_k^\star$ and all $k \in [K]$, it holds with probability at least $1-4N^{-1}$ that 
\begin{align}\label{eq:norm ae}
\left| \|\bm a_i\| - \sqrt{d} \right|  \le \alpha,\quad \left|\|\be_i\| - \delta_k \sqrt{n}\right| \le  \delta_k \alpha. 
\end{align}
(ii) For all $i \in C_k^\star$ and all $k \in [K]$, it holds with probability at least $1-4N^{-1}$ that 
\begin{align}
\left\|\bm U_k^\star \bm a_i\bm e_i^T \right\| \le \delta_k(\sqrt{d}+\alpha)(\sqrt n+\alpha).
\end{align}
(iii) Let $\bm V \in \mO^{n\times d}$ be given for some $l \in [K]$. For all $i \in C_k^\star$ and all $k \in [K]$, it holds with probability at least $1-2N^{-1}$ that 
\begin{align}\label{eq:norm Ua}
\left| \|\bm V^T\bm U_k^\star\bm a_i\| - \|\bm V^T\bm U_k^\star\|_F  \right| \le \alpha.
\end{align}
\end{lemma} 
\begin{proof}
According to Definition \ref{def:MoG} and $i \in C_k^\star$, we have $\bm{z}_i = \bm{U}_k^*\bm{a}_i + \bm{e}_i$, where $\bm{a}_i \sim \mN(\bm{0},\bm{I}_{d})$ is independent of $\bm{e}_i \sim \mN(\bm{0},\delta_k^2\bm{I}_{n})$. 

(i) Applying Lemma \ref{lem:norm gaus} to $\bm{a}_i \sim \mN(\bm{0},\bm{I}_{d})$, together with setting $t=2\sqrt{\log N}$ and $\lambda_j=1$ for all $j\in [d]$, yields  
\begin{align*} 
\P\left(\left|\| \ba_i\| - \sqrt{d}\right| \ge 2\sqrt{\log N} + 2 \right) \le 2N^{-2}.
\end{align*} 
This implies that $\left|\| \bm a_i\| - \sqrt{d}\right| \le \alpha $ holds for all $i\in C_k^\star$ and $k \in [K]$ with probability at least $1-2N^{-1}$. Similarly, we obtain that $\left|\| \be_i\| - \delta_k\sqrt{n}\right| \le \delta_k\alpha $ holds for all $i\in C_k^\star$ and $k \in [K]$ with probability at least $1-2N^{-1}$. Finally, using the union bound, we prove \eqref{eq:norm ae}. 

(ii) On the event in part (i), since \(\bm U_k^\star\) has orthonormal columns,
\[
\left\|\bm U_k^\star \bm a_i\bm e_i^T\right\|
\le \|\bm U_k^\star\bm a_i\|\,\|\bm e_i\|
= \|\bm a_i\|\,\|\bm e_i\|
\le \delta_k(\sqrt d+\alpha)(\sqrt n+\alpha).
\]
The same union-bound probability as in part (i) gives the displayed claim.

(iii) Let $\bV^T\bU^\star_k = \bm P\bm{\Sigma}\bm{Q}^T$ be a singular value decomposition of $\bV^T\bU^\star_k$, where $\bm{\Sigma} \in \R^{d\times d}$ with the diagonal elements $0 \le \sigma_{d} \le \dots \sigma_1 \le 1$ being the singular values of $\bV^T\bU^\star_k$, $\bP  \in \mO^{d}$, and $\bQ  \in \mO^{d}$. This, together with the orthogonal invariance of the Gaussian distribution, yields  
\begin{align}\label{eq1:lem norm Gaus}
\|\bV^T\bU^\star_k\ba_i\| = \|\bm{\Sigma}\bm{Q}^T\ba_i\| \overset{d}{=}  \|\bm{\Sigma}\ba_i\| = \sqrt{\sum_{j=1}^{d} \sigma_j^2a_{ij}^2}. 
 \end{align}
Using Lemma \ref{lem:norm gaus} with setting $t=2\sigma_1\sqrt{\log N}$ and $\lambda_j=\sigma_{j} \le 1$ for all $j$ yields  
\begin{align*}
\P\left( \left|\|\bV^T\bU^\star_k\bm a_i\| - \|\bV^T\bU^\star_k\|_F\right| \ge \sigma_1\alpha \right) & = \P\left( \left|\sqrt{\sum_{j=1}^{d} \sigma_j^2a_{ij}^2} - \sqrt{\sum_{j=1}^{d} \sigma_j^2} \right| \ge \sigma_1\alpha \right)   \le 2N^{-2}. 
\end{align*}
This, together with $\sigma_1 \le 1$ and the union bound, yields \eqref{eq:norm Ua}.  
\end{proof}


\end{document}